\documentclass{amsart}

\usepackage{graphicx}
\usepackage{libertine}
\usepackage{wasysym}

\newtheorem{theorem}{Theorem}[section]
\newtheorem{lemma}[theorem]{Lemma}
\newtheorem{corollary}[theorem]{Corollary}
\newtheorem{prop}[theorem]{Proposition}

\newtheorem{notation}[theorem]{Notation}

\theoremstyle{definition}
\newtheorem{definition}[theorem]{Definition}

\theoremstyle{remark}
\newtheorem{remark}[theorem]{Remark}

\numberwithin{equation}{section}

\newcommand*{\QED}{\hfill\ensuremath{\square}}

\usepackage[top=0.75in, bottom=0.75in, left=0.75in, right=0.75in]{geometry}

\usepackage{mathrsfs}

\usepackage[english]{babel}
\usepackage[utf8x]{inputenc}
\usepackage{amsmath}
\usepackage{graphicx}
\usepackage{amssymb}
\usepackage{amsthm}
\usepackage{fancyhdr}
\usepackage{amsmath}
\usepackage{amsxtra}
\usepackage{amscd}
\usepackage{amsthm}
\usepackage{amsfonts}
\usepackage{amssymb}
\usepackage{eucal}
\usepackage[all]{xy}
\usepackage{graphicx}
\usepackage{pifont}
\usepackage{comment}
\usepackage{tikz}
\usetikzlibrary{matrix,arrows}

\newcommand\nc{\newcommand}
\nc{\on}{\operatorname}
\nc{\R}{\mathbb R}
\nc{\C}{\mathbb C}
\nc{\Q}{\mathbb Q}
\nc{\Z}{\mathbb Z}
\nc{\N}{\mathbb N}
\nc{\F}{\mathbb F}
\nc{\Hom}{\on{Hom}}
\nc{\wt}{\widetilde}
\nc{\kernel}{\text{ker}}
\nc{\image}{\text{Im}}
\nc{\sls}{\subsetneq ... \subsetneq}
\nc{\ssn}{\subsetneq}
\nc{\bull}{$\bullet \, \,$}
\nc{\ol}{\overline}
\nc{\short}[3]{0 \longrightarrow #1 \longrightarrow #2 \longrightarrow #3 \longrightarrow 0}
\nc{\pd}[2]{\frac{\partial #1}{\partial #2}}
\nc{\one}{\mathbf{1}}
\nc{\rnc}{\renewcommand}
\nc{\e}{\varepsilon}
\nc{\DMO}{\DeclareMathOperator}
\nc{\dd}{\emph{d}}
\nc{\grad}{\nabla}

\rnc{\leq}{\leqslant}
\rnc{\geq}{\geqslant}
\rnc{\int}{\varint}
\rnc{\d}{\text{d}}

\DeclareMathOperator{\E}{\mathbb{E}}

\DMO{\D}{\text{D}}
\DMO{\pv}{p.v.}

\newenvironment{nouppercase}{%
  \renewcommand{\uppercasenonmath}[1]{}}{}
\begin{document}


\title{\LARGE Bulk Eigenvalue Correlation Statistics of Random Biregular Bipartite Graphs}

\author{\large Kevin Yang$^{\dagger}$} \thanks{$\dagger$ Stanford University, Department of Mathematics. Email: kyang95@stanford.edu.}

\begin{nouppercase}
\maketitle
\end{nouppercase}
\begin{center}
\today
\end{center}

\begin{abstract}
In this paper, we consider the random matrix ensemble given by $(d_b, d_w)$-regular graphs on $M$ black vertices and $N$ white vertices, where $d_b \in [N^{\gamma}, N^{2/3 - \gamma}]$ for any $\gamma > 0$. We simultaneously prove that the bulk eigenvalue correlation statistics for both  normalized adjacency matrices and their corresponding covariance matrices are stable for short times. Combined with an ergodicity analysis of the Dyson Brownian motion in another paper, this proves universality of bulk eigenvalue correlation statistics, matching normalized adjacency matrices with the GOE and the corresponding covariance matrices with the Gaussian Wishart Ensemble.
\end{abstract}

\tableofcontents

\newpage
\section{Introduction}
The Wigner-Dyson-Gaudin-Mehta conjecture asserts that the local eigenvalue statistics, such as eigenvalue gaps and other statistics on the gap scale, are universal, depending only on the symmetry class of the matrix ensemble (real symmetric, complex Hermitian, quaternion). For Wigner matrices, the eigenvalue statistics are important in probability theory, mathematical physics, and have deep connections with population dynamics, zeros of $L$-functions, etc. For these matrices the WDGM conjecture has been proven for a large class of matrices in \cite{BHKY}, \cite{ESY}, \cite{ESYY}, \cite{EYY}, \cite{HLY}, \cite{LSY}, \cite{LY}, and \cite{LS}, among other papers. These ensembles range from restrictive classical Wigner ensembles to correlated matrices coming from random regular graphs. Here, the limiting universal ensemble is the Gaussian Orthogonal Ensemble, or GOE for short, whose distribution is given as follows:
\begin{align}
W \ = \ (W_{ij})_{i, j = 1}^N, \quad W_{ij} \ \sim \ \mathscr{N}(0, 1/N) \mathbf{1}_{i \neq j} + \mathscr{N}(0, 2/N) \mathbf{1}_{i = j}.
\end{align}
Of course here, we assume a symmetry constraint so that $W_{ij} = W_{ji}$; otherwise the entries are independent. The procedure for universality for Wigner matrices in the papers listed above has been coined the robust three-step strategy and is carried out as follows:
\begin{itemize}
\item Step 1: Derive a local law, showing convergence of the Stieltjes transform of the random matrix ensemble at microscopic scales.
\item Step 2: Show short-time stability for times $t \leq N^{-1 + \e}$ of the eigenvalue statistics under Dyson's Brownian motion (DBM), which stochastically interpolate between the random matrix ensemble of interest and the limiting universal ensemble (e.g. GOE).
\item Step 3: Show a short-time to convergence under the interpolating DBM, i.e. show that after time $t \geq N^{-1 + \delta}$ the eigenvalue statistics of the evolved matrix ensemble agree with those of the limiting ensemble.
\end{itemize}

On the other hand, covariance matrices are another historically fundamental class of random matrix ensembles. Covariance matrices are especially important in high-dimensional data and statistical analysis, with applications in a wide range of disciplines such as population ecology. The spectral statistics of covariance matrices are crucial in a classical, powerful method of statistical analysis known as principal component analysis, or PCA for short, motivating the problem of universality for large covariance matrices from the perspective of any statistical science. For covariance matrices, otherwise known as Wishart matrices, the universality problem has been explored in much less depth. Universality for a rather restrictive class of random matrix ensembles has been proven in \cite{ESY}, \cite{ESYY}, \cite{NP1}, and \cite{NP2}. The techniques in these papers, however, do not extend to studying sparse covariance matrices or correlated covariance matrices; the former exhibits a local law as proven in \cite{A}, but no universality has been proven just yet. 

In this paper, we study covariance matrices arising from the off-diagonal blocks of adjacency matrices of biregular bipartite graphs. This mimics the random regular graph ensemble studied in \cite{BHKY}. In particular, we complete the second step of the robust three-step strategy described above for covariance matrices with correlated data entries and show the eigenvalue correlation statistics agree with those of the Gaussian Wishart ensemble (GWE), given by the following distribution:
\begin{align}
X_W \ = \ W^{\ast} W, \quad  W \ = \ (W_{ij})_{i = 1, j = 1}^{i = M, j = N}, \quad W_{ij} \ \sim \ \mathscr{N}(0, 1/N).
\end{align}
Here, we do not impose a symmetry constraint of the matrix $W$, and all the entries are i.i.d. standard normal random variables. 

In this paper, we simultaneously develop the second step for the honest adjacency matrix, which may be thought of as a linearization of the covariance matrix, comparing this random matrix ensemble to the GOE. This idea is original to this paper to the author's knowledge. Thus, because in \cite{Y1}, the author derives a local law for both the covariance matrix ensemble and the linearized ensemble, this paper may be realized as the second paper in a series of three papers, the third of which is \cite{Y3} and completes step 3 in a vastly general context.
\subsection{Acknowledgements}
The author thanks H.T. Yau and Roland Bauerschmidt for answering the author's questions pertaining to random regular graphs. This paper was written while the author was a student at Harvard University.
\subsection{Notation}
We adopt the Landau notation for big-Oh notation, and the notation $a \lesssim b$. We establish the notation $[[a, b]] := [a,b] \cap \Z$. We let $[E]$ denote the underlying vertex set of a graph $E$. For vertices $v, v' \in E$, we let $vv'$ denote the edge in $E$ containing $v$ and $v'$. For a real symmetric matrix $H$, we let $\sigma(H)$ denote its (real) spectrum.
%
%
%
%
\section{Dyson Brownian motion and the main result}
\subsection{Gaussian measures on Hilbert spaces}
Our construction of the Dyson Brownian motion will proceed in two steps, the first of which is a more general result in defining probability measures on Hilbert spaces. The proof is a standard analysis of a Gaussian integral, so we omit it.
\begin{prop} \label{prop:GaussianHilbertSpaces}
Suppose $\mathscr{H}$ is a finite-dimensional Hilbert space with basis $\{\mathbf{f}_\alpha\}_{\alpha}$. Let $\{z_\alpha\}_{\alpha}$ denote a (finite) collection of scalar-valued Gaussian random variables. Then there exists a Gaussian measure on $\mathscr{H}$ given by the following random vector:
\begin{align}
\omega \ = \ \sum_{\alpha} z_\alpha \mathbf{f}_\alpha
\end{align}
such that the measure induced by $\omega$ is invariant under isometries of $\mathscr{H}$. In particular, the Gaussian measure is independent of the choice of basis $\{ \mathbf{f}_\alpha \}_\alpha$. 
\end{prop}
One consequence of the Gaussian measure is the existence of Brownian motions; for any basis $\{ \mathbf{f}_\alpha \}_\alpha$, we may define the Brownian motion on $\mathscr{H}$ as
\begin{align}
B(t) \ = \ \sum_{\alpha} \ B_\alpha(t) \mathbf{f}_\alpha,
\end{align}
where the $\{ B_\alpha(t) \}_{\alpha}$ are independent standard one-dimensional Brownian motions. Because the Gaussian measure is invariant under isometries of $\mathscr{H}$ by Proposition \ref{prop:GaussianHilbertSpaces}, the Brownian motion $B(t)$ is also invariant under isometries of $\mathscr{H}$.

To begin this second step, we now take the Hilbert space $\mathscr{H} = \mathbf{M}_{M \times N}(\R)$. Here, as in the setting of \cite{Y1}, we assume that $\alpha := M/N \geq 1$ without loss of generality throughout this paper. The inner product structure on this matrix Hilbert space is given by the following trace-pairing/Hilbert-Schmidt norm:
\begin{align}
\langle A, B \rangle \ = \ \on{Tr}(A^{\ast} B). \label{eq:matrixinnerproduct}
\end{align}
We now may define the following Hilbert space whose inner product is induced by the inner product  on $\mathbf{M}_{M \times N}(\R)$:
\begin{align}
\mathscr{M} \ := \ \left\{ X \ = \ \begin{pmatrix} 0 & H \\ H^\ast & 0 \end{pmatrix}: \ H \in \mathbf{M}_{M \times N}(\R) \right\}.
\end{align}
%
\begin{notation}
For convenience, if the blocks of $X \in \mathscr{M}$ are given by $H \in \mathbf{M}_{M \times N}(\R)$, we will write $X = (H, H^\ast)$. 
\end{notation}
We note that (or recall from \cite{Y1}) that adjacency matrices given by $(d_b, d_w)$-regular bipartite graphs share a common eigenvalue-eigenvector pair. Recall from \cite{Y1} the following normalization for these adjacency matrices:
\begin{align}
X(0) \ = \ d_w^{-1/2} \left( \begin{matrix}
0 & H \\
H & 0 \end{matrix} \right), \quad H \ = \ A - \frac{d_b}{N} (\mathbf{1}), \label{eq:normalizedadjacency}
\end{align}
where $(\mathbf{1})$ denotes a matrix whose entries are all $1$ of the appropriate dimension. Here, we use the notation in \cite{Y1} and assume the blocks $A$ and $A^{\ast}$ arise from honest adjacency matrices of biregular bipartite graphs of the following form:
\begin{align}
X_A \ = \begin{pmatrix} 0 & A \\ A^{\ast} & 0 \end{pmatrix}.
\end{align}
In particular, by the Perron-Frobenius theorem, up to a constant factor $d_w^{-1/2}$ the matrix $X(0)$ shares the same eigenvalue-eigenvector pairs as the adjacency matrix $X_A$, except for a one-dimensional eigenspace. On this eigenspace, i.e. that with the constant eigenvector
\begin{align}
\mathbf{e}(i) \ = \ \begin{cases} 
      \frac{1}{\sqrt{M}} & 1 \leq i \leq M \\
      \frac{1}{\sqrt{N}} &  M + 1 \leq i \leq M + N
   \end{cases},
\end{align}
the matrix $X(0)$ has eigenvalue 0. Thus, in terms of studying the nontrivial spectral data of adjacency matrices, it suffices to study the normalized matrix $X(0)$. For a detailed discussion we refer to \cite{Y1}. 
\begin{remark}
As was the convention in \cite{Y1}, we now refer to $(d_b, d_w)$-regular bipartite graphs as biregular bipartite graphs whenever there is no risk of confusion concerning the regularity parameters of the graph.
\end{remark}
In terms of the Dyson Brownian motion dynamics, we account for this algebraic constraint on adjacency matrices, we now instead look at the following matrix space:
\begin{align}
\mathscr{M}_{\mathbf{e}} \ = \ \left\{ X \in \mathscr{M}: \ X \mathbf{e} = 0 \right\}.
\end{align}
As a (closed) subspace of $\mathscr{M}$, the space $\mathscr{M}_{\mathbf{e}}$ inherits a canonical Hilbert space structure. We note, however, that the space $\mathscr{M}_{\mathbf{e}}$ does not come with a convenient set of coordinates with which we can write down stochastic matrix dynamics. To better understand the Gaussian measure on $\mathbf{M}_{\mathbf{e}}$, we appeal to the following parameterization of $\mathscr{M}_{\mathbf{e}}$, which follows from the singular value decomposition (SVD) of the block $H$.
\begin{lemma} \label{lemma:parameterizationvanishingmatrices}
Suppose $X \in \mathscr{M}_{\mathbf{e}}$ has the block representation $X = (H, H^\ast)$. Then, for some matrix $\wt{X} \in \mathbf{M}_{(M-1) \times (N-1)}(\R)$,
\begin{align}
X(0) \ = \ O(M) \times ( \wt{X} \oplus 0) \times O(N)^\ast, \label{eq:parameterizationmematrices}
\end{align}
where the matrices $O(M), O(N)$ are orthogonal of dimension $M$ and $N$, respectively. Here, multiplication on the RHS is multiplication as matrices. Moreover, under the induced map
\begin{align}
\mathscr{M}_{\mathbf{e}} \ \to \ \mathbf{M}_{(M-1) \times (N-1)}(\R),
\end{align}
the Gaussian measure is invariant, where the latter space is equipped with the same inner product given by \eqref{eq:matrixinnerproduct}.
\end{lemma}
\begin{proof}
It remains to show the SVD preserves the Hilbert space inner product on $\mathscr{M}_{\mathbf{e}}$. This follows from the following straightforward calculation for matrices $A, B \in \mathscr{M}_{\mathbf{e}}$:
\begin{align}
\on{Tr} (A^\ast B) \ &= \ \on{Tr} \left(O(N) \left( \wt{A}^\ast \oplus 0 \right) O(M)^\ast O(M) \left( \wt{B} \oplus 0 \right) O(N)^\ast \right) \\
&= \ \on{Tr} \left( O(N) \left( \wt{A}^\ast \wt{B} \oplus 0 \right) O(N)^\ast \right) \\
&= \ \on{Tr} \left( \wt{A}^\ast \wt{B} \right).
\end{align}
\end{proof}
Thus, up to the change of coordinates given by the SVD in Lemma \ref{lemma:parameterizationvanishingmatrices}, it suffices to study the Gaussian measure on the smaller matrix space $\mathbf{M}_{(M-1) \times (N-1)}(\R)$. This comes with a standard choice of coordinates. 

We conclude this discussion of Gaussian measures on the matrix space $\mathscr{M}_{\mathbf{e}}$ by writing down the following matrix-valued SDE dynamics of primary interest:
We now proceed to introduce the following stochastic differential equation (SDE) as motivated by earlier works in universality:
\begin{align}
\d X(t) \ = \ \frac{1}{\sqrt{N}} \d\mathbf{B}(t) \ - \ \frac{1}{2} X(t) \ \d t, \label{eq:matrixOUprocess}
\end{align}
where $\mathbf{B}(t)$ denotes the standard Brownian motion on $\mathscr{M}_{\mathbf{e}}$. In particular, up to a SVD change of coordinates, the Brownian motion is equal, in law, to the entry-wise standard one-dimensional Brownian motion on $\mathbf{M}_{(M-1) \times (N-1)}(\R)$. Although $X(t)$ contains only $N-1$ nontrivial eigenvalues, we retain the normalization $N^{-1/2}$ in the SDE \eqref{eq:matrixOUprocess}.
\subsection{Analysis of switchings on the DBM}
We now briefly review the notion of switchings of biregular bipartite graphs as discussed in detail in \cite{Y1}. 
Suppose $E$ is a biregular bipartite graph, and consider a pair of edges $e_1 = ij, e_2 = mn \in E$ with four distinct vertices. We let $E_{e_1, e_2}$ denote the subgraph of $E$ with edges $e_1, e_2$. 
\begin{definition}
We say a \emph{simple switching} of $E$ at $E_{e_1, e_2}$ is the following perturbed graph:
\begin{align}
E_s \ = \ E - E_{e_1, e_2} + E_{in, mj},
\end{align}
where the operations on the RHS are understood in the sense of adjacency matrices. 
\end{definition}
We note that the switched graph $E_s$ suppresses from its notation the dependence of $e_1, e_2$. This will not be important, but we emphasize it now for clarity's sake.

As in \cite{Y1}, we now interpret the combinatorics of switchings in terms of the corresponding adjacency matrices. To do so, we first define the following matrix:
\begin{align}
\xi_{ij}^{mn} \ = \ \Delta_{ij} + \Delta_{mn} - \Delta_{in} - \Delta_{mj},
\end{align}
where $\Delta_{xy}$ denotes the adjacency matrix of the graph whose only edge is $xy$. With this notation, we may easily deduce
\begin{align}
A(E_s) \ = \ A(E) + \xi_{ij}^{mn},
\end{align}
where $A(E_s)$ (resp. $A(E)$) denotes the adjacency matrix of the graph $E_s$ (resp. $E$). We now define the following:
\begin{align}
X_{ij}^{mn} \ &= \ \on{Tr} \left( \xi_{ij}^{mn} X \right) \ = \ 2 \left( X_{ij} + X_{mn} - X_{in} - X_{mj} \right), \label{eq:switchmatrix} \\
\partial_{ij}^{mn} \ &= \ \on{Tr} \left( \xi_{ij}^{mn} \partial \right) \ = \ 2 \left( \partial_{ij} + \partial_{mn} - \partial_{in} - \partial_{mj} \right). \label{eq:switchpartial}
\end{align}
We now formally gather these terms in the following set:
\begin{align}
\mathscr{X} \ := \ \bigcup_{(i,j) \in V_b} \ \bigcup_{(m,n) \in V_b} \ \{ \xi_{ij}^{mn}\}.
\end{align}
The terms \eqref{eq:switchmatrix} and \eqref{eq:switchpartial} will show up in studying the generator of the process \eqref{eq:matrixOUprocess}. We make this precise in the following result.
\begin{prop} \label{prop:generatormatrixOUprocess}
The generator of the OU process on $\mathscr{M}_{\mathbf{e}}$ is given by
\begin{align}
\mathscr{L} \ = \ \frac{1}{8MN^2} \sum_{i,j,k,\ell} \ \left( \partial_{ij}^{k\ell} \right)^2 \ - \ \frac{1}{16MN} \sum_{i,j,k,\ell} \ X_{ij}^{k \ell} \partial_{ij}^{k \ell}.
\end{align}
I.e., for any $F \in C^2(\mathscr{M}_{\mathbf{e}})$, we have
\begin{align}
\partial_t \E F(X(t))  \ = \ \E (\mathscr{L} F)(X(t)).
\end{align}
\end{prop}
The proof of Proposition \ref{prop:generatormatrixOUprocess} will amount to an application of the Ito formula. To apply this formula, however, we need to compute the quadratic covariations of the matrix entries $\d X_{ij}$, where we choose coordinates as in Lemma \ref{lemma:parameterizationvanishingmatrices}. Before we compute these quadratic covariations, we establish the following convention of black and white vertices from \cite{Y1}.
\begin{notation}
Define $V_b$ to be the set of indices $\{ (i,j) \}$ such that $i, j-M \in [1, M]$. Similarly, define $V_w$ to be those indices $(k, \ell)$ such that $\ell, k - M \in [1, M]$.
\end{notation}
\begin{lemma} \label{lemma:quadraticvariation}
Fix any two indices $(i,j)$ and $(k, \ell)$. Then, we have
\begin{align}
\emph{d} \langle X_{ij}, X_{k \ell} \rangle \ = \ \left\{
     \begin{array}{@{}l@{\thinspace}l}
       \frac{1}{N} \left( \delta_{ik} - \frac{1}{M} \right) \left( \delta_{j \ell} - \frac{1}{N} \right)		&: \quad (i,j), (k, \ell) \in V_b; \\       
       \frac{1}{N} \left( \delta_{ik} - \frac{1}{N} \right) \left( \delta_{j \ell} - \frac{1}{M} \right)		&: \quad (i,j), (k,\ell) \in V_w; \\
       \frac{1}{N} 	\left( \delta_{i \ell} - \frac{1}{M} \right) \left( \delta_{jk} - \frac{1}{N} \right)		&: \quad (i,j) \in V_b, (k, \ell) \in V_w; \\
       \frac{1}{N} 	\left( \delta_{i \ell} - \frac{1}{N} \right) \left( \delta_{jk} - \frac{1}{M} \right)		&: \quad (i,j) \in V_w, (k, \ell) \in V_b. \ 
     \end{array}
   \right.
\end{align}
\end{lemma}
\begin{proof}
We consider the case $(i,j), (k, \ell) \in V_b$; the other cases follow similarly. By Lemma \ref{lemma:parameterizationvanishingmatrices}, we may assume the normalized adjacency matrix $X$ is of the form 
\begin{align}
X = O(M) \left( \wt{X} \oplus 0 \right) O(N)^\ast
\end{align}
where the orthogonal component corresponds to the span of the eigenvector $\mathbf{e}$ and is thus constant in time. Moreover, we may assume the Gaussian measure on $\mathbf{M}_{(M-1) \times (N-1)}(\R)$ is given by drawing each entry from independent standard one-dimensional Gaussian distributions. This implies the quadratic covariation process for $\wt{X}(t)$ is given by
\begin{align}
\d \langle \wt{X}_{ij}(t), \wt{X}_{k\ell}(t) \rangle \ = \ \frac{1}{N} \delta_{ik} \delta_{j \ell}
\end{align}
where we use the assumption $(i, j), (k, \ell) \in V_b$. Thus, under this same assumption, we compute
\begin{align}
\d \langle H_{ij}, H_{k \ell} \rangle \ &= \ \sum_{m,n} \sum_{x,y} \ \left[ O(M)_{im} O(N)_{jn} O(M)_{kx} O(N)_{\ell y} \right] \  \d \langle \wt{X}_{mn}, \wt{X}_{xy} \rangle \\
&= \ \sum_{m,n} \sum_{x,y} \ \left[ O(M)_{im} O(N)_{jn} O(M)_{kx} O(N)_{\ell y} \right] \times \frac{1}{N} \delta_{mx} \delta_{n y} \\
&= \ \frac{1}{N} \sum_{m,n} \ O(M)_{im} O(N)_{jn} O(M)_{km} O(N)_{\ell n} \\
&= \ \frac{1}{N} \left( \sum_{m = 1}^{M-1} \ O(M)_{im} O(M)_{km} \right) \left( \sum_{n = 1}^{N-1} \ O(N)_{jn} O(N)_{\ell n} \right).
\end{align}
Because the matrices $O(M), O(N)$ give the SVD of the matrix $H(t)$, we know $O(M)_{iM} = M^{-1/2}$ and $O(N)_{kN} = N^{-1/2}$ for any suitable indices $i, k$. With this and the assumption that $O(M), O(N)$ are orthogonal matrices, we see
\begin{align}
\sum_{m = 1}^{M-1} \ O(M)_{im} O(M)_{km} \ &= \ \delta_{ik} - \frac{1}{M}, \\
\sum_{n = 1}^{N-1} \ O(N)_{jn} O(N)_{\ell n} \ &= \ \delta_{j \ell} - \frac{1}{N},
\end{align}
which completes the derivation of the covariation processes in the case $(i, j), (k, \ell) \in V_b$. 
\end{proof}
Before proving Proposition \ref{prop:generatormatrixOUprocess}, we first establish the following shorthand for differentiation of any sufficiently smooth function $F$ on a matrix space and indices $(i,j)$:
\begin{align}
\partial_{ij} F(X) \ := \ \left[ \partial_{X_{ij}} F \right] (X).
\end{align}
We now proceed with the proof of Proposition \ref{prop:generatormatrixOUprocess}. By the Ito formula, for any $F \in C^2(\mathscr{M}_{\mathbf{e}})$, we have
\begin{align}
\d F(X) \ = \ \sum_{i, j = 1}^{M+N} \ \left[ \partial_{ij} F \right](X) \ \d X_{ij} \ + \ \frac12 \sum_{i,j,k,\ell = 1}^{M+N} \ \left[ \partial_{ij} \partial_{k \ell} F \right] (X) \ \d \langle X_{ij}, X_{k \ell} \rangle.
\end{align}
Taking expectation, the martingale term in $\d X_{ij}$ vanishes; by construction, we thus have
\begin{align}
\d \E F(X) \ = \ - \frac12 \sum_{i, j = 1}^{M+N} \ \E \left[ X_{ij} \left[ \partial_{ij} F \right] (X) \right] \d t \ + \ \frac12 \sum_{i,j,k, \ell = 1}^{M+N} \ \E \left[ \left[ \partial_{ij} \partial_{k \ell} F \right] (X) \right] \d \langle X_{ij}, X_{k \ell} \rangle.
\end{align}
Computing the quadratic covariation terms via Lemma \ref{lemma:quadraticvariation}, we have
\begin{align}
MN^2 \sum_{i,j,k, \ell = 1}^{M+N} \ \E \left[ \left[ \partial_{ij} \partial_{k \ell} F \right] (X) \right] \d \langle X_{ij}, X_{k \ell} \rangle \ = &\ \sum_{(i,j), (k, \ell) \in V_b} \ \E \left[ [\partial_{ij} (\partial_{ij} + \partial_{k \ell} - \partial_{i \ell} - \partial_{jk}) F](X)\right] \d t \label{eq:bsame} \\
&+ \ \sum_{(i,j), (k, \ell) \in V_w} \ \E \left[ [\partial_{ij}(\partial_{ij} + \partial_{k \ell} - \partial_{i \ell} - \partial_{jk}) F](X) \right] \d t \label{eq:wsame} \\
&+ \ \sum_{(i,j) \in V_b, (k, \ell) \in V_w} \ \E \left[ [\partial_{ij}(\partial_{ij} + \partial_{k \ell} - \partial_{ik} - \partial_{j \ell} )F ](X)  \right] \d t \label{eq:bw} \\
&+ \ \sum_{(i,j) \in V_w, (k, \ell) \in V_b} \ \E \left[ [\partial_{ij}(\partial_{ij} + \partial_{k \ell} - \partial_{ik} - \partial_{j \ell} )F ](X) \right] \d t. \label{eq:wb}
\end{align}
We note however, upon the bijection $(i,j) \mapsto (k, \ell)$, that the summations given by the RHS of \eqref{eq:bsame} and \eqref{eq:wsame} are equal. Similarly, we see the summations given by \eqref{eq:bw} and \eqref{eq:wb} are also equal. Lastly, we see \eqref{eq:bsame} and \eqref{eq:bw} are equal upon switching the indices $k, \ell$. Thus, because the process $\E F(X)$ contains no martingale term,
\begin{align}
\mathscr{L} \ &= \ \frac{2}{MN^2} \sum_{(i,j), (k,\ell) \in V_b} \ \partial_{ij} \left( \partial_{ij} + \partial_{k\ell} - \partial_{i \ell} - \partial_{jk} \right) \ - \ \frac{1}{2} \sum_{i,j = 1}^{M+N} \ X_{ij} \partial_{ij} \\
&= \ \frac{2}{MN^2} \sum_{(i,j), (k,\ell) \in V_b} \ \partial_{ij} \left( \partial_{ij} + \partial_{k \ell} - \partial_{i \ell} - \partial_{jk} \right) \ - \ \sum_{(i,j) \in V_b} \ X_{ij} \partial_{ij}, \label{eq:normalizedgenerator}
\end{align}
where the second equality \eqref{eq:normalizedgenerator} holds since $X$ is symmetric. To understand the second-order terms in \eqref{eq:normalizedgenerator}, we claim
\begin{align}
\sum_{(i,j),(k,\ell) \in V_b} \ \partial_{ij} \left( \partial_{ij} + \partial_{k \ell} - \partial_{i \ell} - \partial_{jk} \right) \ = \ \frac{1}{4} \sum_{(i,j), (k, \ell) \in V_b} \ \left( \partial_{ij} + \partial_{k \ell} - \partial_{i \ell} - \partial_{jk} \right)^2. \label{eq:secondordertermsgrouped}
\end{align}
Indeed, \eqref{eq:secondordertermsgrouped} follows from the fact that we are summing over all indices $(i,j), (k, \ell) \in V_b$. For the same reason, as well as the assumed relations $\sum_j X_{ij} = \sum_i X_{ij} = 0$, we also have
\begin{align}
\sum_{(i,j) \in V_b} \ X_{ij} \partial_{ij} \ = \ \frac{1}{4MN} \sum_{(i,j), (k,\ell) \in V_b} \ \left( X_{ij} + X_{k \ell} - X_{i \ell} - X_{jk} \right) \left( \partial_{ij} + \partial_{k\ell} - \partial_{i \ell} - \partial_{jk} \right).
\end{align}
This completes the proof of Proposition \ref{prop:generatormatrixOUprocess}. \QED
\subsection{The main result: bulk correlation statistics}
\begin{definition} \label{definition:averagedcorrelation}
Suppose $H_1$ and $H_2$ are two random matrix ensembles of equal dimension $N$ in a matrix space $\mathbf{M}$, e.g. Wigner matrices, Wishart matrices, and the space $\mathscr{M}_{\mathbf{e}}$. We say the \emph{averaged bulk eigenvalue correlation statistics} of $H_1$ and $H_2$ coincide at the energy $E_0$ if the following holds.

For any $n \in \Z_{\geq 0}$, any test function $\varphi \in C_c^{\infty}(\R)$, and a constant $c > 0$ sufficiently small, we have for $b = N^{-1 + c}$
\begin{align}
\frac{1}{2b} \int_{E_0 - b}^{E_0 + b} \ \d E' \ \int_{\R^n} \ \varphi(x_1, \ldots, x_n) N^n \left( \varrho_{H_1}^{(n)} - \varrho_{H_2}^{(n)} \right) \left( E' + \frac{dx_1}{N \varrho_{\infty}(E_0)}, \ldots, E' + \frac{dx_n}{N \varrho_{\infty}(E_0)} \right) \ = \ o_{N \to \infty}(1)
\end{align}
where $\varrho_{H_i}^{(n)}$ denotes the $n$-point correlation function of the matrix ensemble $H_i$ (for $i = 1, 2$). Here, we also use $\varrho_{\infty}$ to denote the density function of either the Marchenko-Pastur law or the semicircle law depending on if the random matrix ensembles $H_1, H_2$ are covariance matrix ensembles or Wigner matrix ensembles, respectively. 
\end{definition}
In particular, Definition \ref{definition:averagedcorrelation} requires a small average around the energy $E_0$. We note that there are results, e.g. in \cite{LSY} in the ensemble of Wigner matrices, that provide similar results without an average of the energy; universality results along this line are known as \emph{fixed energy} universality results. Although it is believed that the arguments in \cite{LSY} extend to linearized covariance matrices, we do not pursue fixed energy universality in this paper.

To state the main theorem, we now introduce the following notation for the covariance matrix ensembles (and their linearization ensembles) at a given time. This is the notation used in \cite{Y1}, for example.
\begin{notation}
For a given time $t \geq 0$, we let $\mathscr{X}_{\ast}(t)$ denote the random matrix ensemble of matrices $X_\ast(t) = H(t)^\ast H(t)$, where the matrix $H(t)$ solves the matrix-valued Ornstein-Uhlenbeck equation 
\begin{align}
\emph{d}H(t) \ = \ \frac{1}{\sqrt{N}} \emph{d}\mathbf{B}(t) \ - \ \frac12 H(t) \ \emph{d} t \label{eq:matrixOUblock}
\end{align}
with initial data $H(0)$ the upper-right block of the normalized adjacency matrix $X(0) = (H(0), H^\ast(0))$ of a graph in $\Omega$.

Similarly, we let $\mathscr{X}(t)$ denote the random matrix ensembles of linearizations $X(t) = (H(t), H(t)^\ast)$.
\end{notation}
We now define the bulk of $\varrho_{\on{MP}}$ as the following interval:
\begin{align}
\mathscr{I}_{\on{MP}, \e} \ = \ \left[ \e, \ (1-\e)(1 + \sqrt{\gamma})^2 \right].
\end{align}
Here, $\e > 0$ is a fixed (small) constant as in the definition of the domains $U_{\e}$ and $U_{\e, \delta}$. We also recall the definition $\gamma := 1/\alpha$ from \cite{Y1}. Similarly, we define the bulk of the linearization to be 
\begin{align}
\mathscr{I}_{\on{linear},\e} \ = \ \pm \sqrt{\mathscr{I}_{\on{MP},\e}} \ = \ \pm \left[ \sqrt{\e}, \sqrt{1 - \e}(1 + \sqrt{\gamma}) \right]
\end{align}
Lastly, we introduce the following sparsity parameter, which is the same sparsity parameter used in \cite{Y1}:
\begin{align}
D \ = \ d_b \wedge \frac{N^2}{d_b^3}.
\end{align}
In this thesis, we will let $d_b$ grow as
\begin{align}
N^{\gamma} \ \leq \ d_b \ \leq \ N^{2/3 - \gamma}
\end{align}
for any $\gamma > 0$. In particular, we have the a priori estimate $D \geq N^{\gamma}$.

We now state the main theorem, which serves as the second step in the three-step strategy discussed in the introduction.
\begin{theorem} \label{theorem:averagedcorrelation}
Suppose $\e > 0$ and $\zeta > 0$ are fixed constants. Then, for any $t \in [0, N^{-1 - \zeta} D^{1/2}]$ and any energy $E \in \mathscr{I}_{\on{linear}, \e}$, the averaged bulk eigenvalue correlation statistics of $\mathscr{X}(0)$ and $\mathscr{X}(t)$ coincide.

Thus, for any $t \in [0, N^{-1 - \zeta} D^{1/2}]$ and any energy $E \in \mathscr{I}_{\on{MP},\e}$, the averaged bulk correlation statistics of $\mathscr{X}_{\ast}(0)$ and $\mathscr{X}_{\ast}(t)$ coincide.
\end{theorem}
We now briefly remark on the short-time nature of the result in Theorem \ref{theorem:averagedcorrelation}. In particular, because it does not see long-term behavior, Theorem \ref{theorem:averagedcorrelation} does not give full bulk universality of correlation functions. The time interval $[N^{-1 - \zeta} D^{1/2}, \infty)$ is addressed in the following result, which is an immediate consequence of the main theorem in \cite{Y3}. This resembles the long-term behavior results in \cite{LSY} and \cite{LY} applied to the ensemble of random regular graphs in \cite{BHKY}.
\begin{theorem}
For any $t \in [N^{-1 - \zeta} D^{1/2}, \infty)$, the averaged bulk eigenvalue correlation statistics of $\mathscr{X}(t)$ and the GOE coincide. Thus, averaged bulk correlation statistics of $\mathscr{X}_{\ast}(t)$ and the Gaussian Wishart Ensemble (GWE) coincide.
\end{theorem}
%
%
%
\section{Short-Time Stability of DBM}
In this section, we aim to use the combinatorial structure of the generator $\mathscr{L}$ of the process \eqref{eq:matrixOUprocess} to study short-time stability of eigenvalue statistics. To state the short-time stability estimates, we first define the following deterministic semi-norm on suitably regular functions:
\begin{align}
\| F \|_{r, t} \ = \ \left( \E |F(X(t))|^r \right)^{1/r}, \label{eq:classicalseminorm}
\end{align}
where the expectation is taken over the randomness of $X(t)$. We extend this norm to derivatives as follows:
We extend this seminorm for derivatives: for any $F \in C^k(\mathscr{M}_{\mathbf{e}})$:
\begin{align}
\| \partial^k F \|_{r, t} \ := \ \left\| \sup_{\theta_i \in [0,1]} \sup_{X_i \in \mathscr{X}} \  \left| \partial_{X_1} \ldots \partial_{X_k} F \left(\cdot + d_b^{-1/2} \sum_{i = 1}^k \ \theta_i X_i \right) \right| \right\|_{r,t}. \label{eq:seminormderivatives}
\end{align}
In particular for $k = 0$ the seminorms \eqref{eq:classicalseminorm} and \eqref{eq:seminormderivatives} coincide. 

We may now state the stability estimate. All adjacency matrices are of bipartite graphs with the normalization as in \eqref{eq:normalizedadjacency}.
\begin{prop} \label{prop:shorttimeestimates}
Suppose $X(t)$ solves the SDE \eqref{eq:matrixOUprocess} with initial condition $X(0)$ a normalized adjacency matrix. Moreover, suppose $r(\varepsilon)$ is sufficiently large as a function of a fixed $\varepsilon > 0$. Then, for any $F \in C^4(\mathscr{M}_{\mathbf{e}})$, we have
\begin{align}
\E F(X(t)) - \E F(X(0)) \ = \ O \left( D^{-1/2} N^{1 + \varepsilon} \max_{1 \leq i \leq 4} \ \int_0^t \ \| \partial^i F \|_{r,s} \ \emph{d} s \right).
\end{align}
\end{prop}
\subsection{The Poisson process}
To prove Proposition \ref{prop:shorttimeestimates}, we use the combinatorial interpretation of the generator $\mathscr{L}$ to compare it to the generator of a discrete Poisson process on the space of graphs itself. We define this Poisson process via its generator, which acts on functions $f$ on the space of \emph{unnormalized} adjacency matrices $A$ of biregular bipartite graphs:
\begin{align}
\mathcal{Q}f(A) \ = \ \frac{1}{4Nd_w} \ \sum_{(i,j) \in V_b} \ \sum_{(m,n) \in V_b} \ I_{ij}^{mn}(A) \left[ f(A - \xi_{ij}^{mn}) - f(A) \right],
\end{align}
where the coefficients are defined by
\begin{align}
I_{ij}^{mn}(A) \ = \ A_{ij} A_{mn} \left( 1 - A_{in} \right) \left(1 - A_{mj} \right).
\end{align}
The coefficient $I_{ij}^{mn}(A)$ detects whether or not the 1-regular graph $\{ ij, mn\}$ embeds, graph-theoretically, into the graph corresponding to the adjacency matrix $A$. Thus, by the basic theory of Poisson processes, the process corresponding to the generator $\mathcal{Q}$ may be described as follows: a Poisson clock of rate depending on $N, d_w$ signals events, which are performing a switching as described earlier in this paper about a randomly chosen 1-regular subgraph of $A$. With this description, as in \cite{Y1}, we deduce the following dynamical result concerning $\mathcal{Q}$. First, we adopt the notation in \cite{Y1} and let $\Omega$ denote the set of biregular bipartite graphs.
\begin{prop} \label{prop:invariantuniformgeneratorvanishes}
Let $\mu_{\on{unif}}$ denote the uniform probability measure on $\Omega$. Then, $\mu_{\on{unif}}$ is invariant under the generator $\mathcal{Q}$.
\end{prop}
In order to compare the generators $\mathcal{Q}$ and $\mathscr{L}$, we first note that they act on different function spaces; $\mathscr{Q}$ acts on functions of unnormalized matrices, whereas $\mathscr{L}$ acts on functions of normalized matrices. Because $\mathscr{L}$ is a second-order differential operator, to account for the Poisson clock rate in $\mathcal{Q}$, given the normalization in \eqref{eq:normalizedadjacency} we define, for any $F \in C^4(\mathscr{M}_{\mathbf{e}})$, the function $f: \Omega \to \R$ given by
\begin{align}
f(X_A) \ = \ f_F(X_A) \ = \ F(X).
\end{align}
%
\begin{remark}
We briefly remark here that the a priori $C^4$-regularity on the function $F$ is an unnecessary assumption in establishing the above convention. However, our short-time stability results depend on this assumed regularity, so we will assume $F$ is $C^4$ throughout for sake of simplicity. 
\end{remark}
We now record the main comparison estimate for the generator $\mathcal{Q}$ and $\mathscr{L}$.
\begin{prop} \label{prop:comparegeneratorsdeterministic}
Fix $\varepsilon > 0$ and $r(\varepsilon)$ sufficiently large depending on $\varepsilon$. For any $F \in \mathscr{C}^4(\mathscr{M}_{\mathbf{e}})$ and any adjacency matrix $X_A$ and its normalization $X$, we have
\begin{align}
\mathcal{Q}f_F(X_A) \ = \ \mathscr{L}F(X) + R, 
\end{align}
where the error term $R$ satisfies the following bound in expectation:
\begin{align}
\E_{\mu_{\on{unif}}} R \ = \ O \left( D^{-1/2} N^{1 + \varepsilon} \right) \max_{1 \leq i \leq 4} \ \| \partial^i F \|_{r(\varepsilon), 0}.
\end{align}
I.e., the expectation is taken over the randomness of the uniform probability measure on $\Omega$.
\end{prop}
We now take Proposition \ref{prop:comparegeneratorsdeterministic} for granted and derive the estimate in Proposition \ref{prop:shorttimeestimates}. To this end, we note that, in law, we may parameterize $X(t)$ as
\begin{align}
X(t) \ \overset{d}= \ e^{-t/2} X(0) + (1 - e^{-t})^{1/2} W,
\end{align}
where $X(t)$ solves the SDE \eqref{eq:matrixOUprocess} and $W$ is a Gaussian matrix according to the Gaussian measure on $\mathscr{M}_{\mathbf{e}}$. Here, $W$ and $X(0)$ are independent. We now define the following functions which allow us to condition on the two ingredients $X(0)$ and $W$, respectively:
\begin{align}
F_X(W) \ := \ F_W(X) \ := \ F \left( e^{-t/2} X(0) + (1 - e^{-t})^{1/2} W \right).
\end{align}
We now decompose the action of $\mathscr{L}$ in terms of an action on $X(0)$ and an action on $W$ as follows.
\begin{lemma} \label{lemma:generatorLdecomposition}
For any $F \in C^2(\mathscr{M}_{\mathbf{e}})$, we have the following decomposition of the action of the generator $\mathscr{L}$:
\begin{align}
\mathscr{L} F(X(t)) \ = \ \mathscr{L} F_X(W) \ + \ \mathscr{L} F_W(X), \label{eq:generatorLdecomposition}
\end{align}
where $\mathscr{L}$ is the generator associated to the matrix-valued Ornstein-Uhlenbeck process with initial condition $X = X(0)$.
\end{lemma}
\begin{proof}
Recall the generator $\mathscr{L}$ is given by the following second-order differential operator:
\begin{align}
\mathscr{L} \ = \ \frac{1}{8 MN^2} \sum_{(i,j),(k,\ell) \in V_b} \ \left( \partial_{ij}^{k \ell} \right)^2 \ - \ \frac{1}{16 MN} \sum_{(i,j),(k,\ell) \in V_b} \ X_{ij}^{k \ell} \partial_{ij}^{k \ell}.
\end{align}
We first address the second-order terms; in particular, for any fixed pair $(i,j)$, we have
\begin{align}
\partial_{ij} \partial_{k \ell} F \left( e^{-t/2} X \ + \ (1 - e^{-t})^{1/2} W \right) \ &= \ \partial_{ij} \left[ e^{-t/2} \partial_{X_{k \ell}} F \ + \ (1 - e^{-t})^{1/2} \partial_{W_{k \ell}} F \right] \\
&= \ e^{-t} \partial_{X_{ij}} \partial_{X_{k \ell}} F \ + \ (1 - e^{-t}) \partial_{W_{ij}} \partial_{W_{k \ell}} F, \label{eq:secondorderdecomposition}
\end{align}
where \eqref{eq:secondorderdecomposition} follows from the independence of $X(0)$ and $W$. However, \eqref{eq:secondorderdecomposition} is exactly equal to
\begin{align}
\partial_{ij} \partial_{k \ell} F_X(W) \ + \ \partial_{ij} \partial_{k \ell} F_W(X).
\end{align}
Thus, the second-order terms on the LHS and RHS, respectively, of \eqref{eq:generatorLdecomposition} agree. To show the first-order terms agree, this amounts to similar identities given as follows:
\begin{align}
\partial_{ij} F(X(t)) \ = \ e^{-t/2} \partial_{ij} F_W(X) \ + \ (1 - e^{-t})^{1/2} F_X(W).
\end{align}
In particular, averaging over indices $(i,j), (k,\ell) \in V_b$ leads to the vanishing of terms of the form $W_{ij}\partial_{X_{ij}}$ and $X_{ij} \partial_{W_{ij}}$ coming from the first-order terms in $\mathscr{L}$.
\end{proof}
Via Lemma \ref{lemma:generatorLdecomposition}, we find
\begin{align}
\E \mathscr{L} F(X(t)) \ = \ \E \mathscr{L} F_X(W) \ + \ \E \mathscr{L} F_W(X),
\end{align}
where the expectation is taken with respect to the randomness of both $X(0)$ and $W$. Differentiating and using Propositions \ref{prop:invariantuniformgeneratorvanishes} and \ref{prop:comparegeneratorsdeterministic}, we find
\begin{align}
\frac{\d}{\d t} \E F(X(t)) \ = \ \E \mathscr{L} F(X(t)) \ &= \ \E \mathscr{L} F_W(X) \\
&= \ \E \mathcal{Q} f_F(A) \ + \ O \left( D^{-1/2} N^{1 + \e} \right) \max_{1 \leq i \leq 4} \ \| \partial^i F \|_{r(\e), 0} \\
&= \ O \left( D^{-1/2} N^{1 + \e} \right) \max_{1 \leq i \leq 4} \ \| \partial^i F \|_{r(\e),0}.
\end{align}
Integrating, we deduce Proposition \ref{prop:shorttimeestimates}. \QED
\subsection{Proof of Proposition \ref{prop:comparegeneratorsdeterministic}}
It now remains to prove Proposition \ref{prop:comparegeneratorsdeterministic}. This follows the arguments in \cite{BHKY}, but we include the discussion for sake of completeness addressing the differences arising from the combinatorial nature of biregular bipartite graphs.

We Taylor expand the function $f: \Omega \to \R$ in $A \in \Omega$ along the directions in $\mathscr{X}$, which we recall is defined as
\begin{align}
\mathscr{X} \ := \ \bigcup_{(i,j) \in V_b} \ \bigcup_{(m,n) \in V_b} \ \{ \xi_{ij}^{mn}\}.
\end{align}
In particular, we will Taylor expand to fourth-order, which is where the a priori regularity $F \in C^4(\mathscr{M}_{\mathbf{e}})$ is assumed. Approximately upon Taylor expanding, we expect
\begin{align}
\mathcal{Q}f(A) \ \approx \ \frac{1}{4Nd_w} \sum_{(i, j) \in V_b} \ \sum_{(m,n) \in V_b} \ A_{ij} A_{mn} \left[ - \partial_{ij}^{mn} f(A) + \frac12( \partial_{ij}^{mn})^2 f(A)\right]. \label{eq:heuristicapproximationgenerator}
\end{align}
We will rewrite the entries of the adjacency matrix as follows: 
\begin{align}
A_{ij} \ &= \ \frac{d_b}{N} + \left( A_{ij} - \frac{d_b}{N} \right), \label{eq:efexpansionij} \\
A_{mn} \ &= \ \frac{d_b}{N} + \left( A_{mn} - \frac{d_b}{N} \right).\label{eq:efexpansionmn}
\end{align}
Plugging these expansions into the heuristic approximation \eqref{eq:heuristicapproximationgenerator}, we obtain the generator $\mathscr{L}$ on the RHS upon the convention $f_F(A) = F(H)$ for any suitably regular function $F \in C^4(\mathscr{M}_{\mathbf{e}})$, in addition to other terms we will show are, in fact, error terms. To make the Taylor expansion \eqref{eq:heuristicapproximationgenerator} more precise, we write
\begin{align}
\mathcal{Q}f(A) \ = \ \frac{1}{4Nd_w} \sum_{(i,j) \in V_b} \ \sum_{(m,n) \in V_b} \ A_{ij} A_{mn} \left[ - \partial_{ij}^{mn} f(A) + \frac12 (\partial_{ij}^{mn})^2 f(A) \right] \ + \ N^2(R_1 + R_2), \label{eq:taylorapproximationgenerator}
\end{align}
where we make the approximation \eqref{eq:heuristicapproximationgenerator} by introducing the following error terms:
\begin{align}
R_1 \ &= \ O \left( \frac{N}{d_w} \right) \frac{1}{N^4} \sum_{(i,j) \in V_b} \ \sum_{(m,n) \in V_b} \ A_{ij} A_{mn} (1 - I(ij, mn)) \sup_{\theta \in [0,1]} \sup_{X \in \mathscr{X}} \ \left| \partial_X f(A + \theta X) \right|, \\
R_2 \ &= \ O \left( \frac{N}{d_w} \right) \frac{1}{N^4} \sum_{(i,j) \in V_b} \ \sum_{(m,n) \in V_b} \ A_{ij} A_{mn} \sup_{\theta_i \in [0,1]} \sup_{X_i \in \mathscr{X}} \ \left| \partial_{X_1} \partial_{X_2} \partial_{X_3} f \left(A + \sum_{i = 1}^3 \ \theta_i X_i \right) \right|.
\end{align}
We proceed with using the representations of the adjacency matrix entries given in \eqref{eq:efexpansionij} and \eqref{eq:efexpansionmn}. In particular, given we expect the fluctuation $A_{ij} - d_b/N$ to be small, we use these representations and study the first-order term in \eqref{eq:taylorapproximationgenerator} by grouping terms as follows:
\begin{align}
\sum_{(i,j),(m,n) \in V_b} \ A_{ij} A_{mn} \partial_{ij}^{mn} f(A) \ = \  &\sum_{(i,j) \in V_b} \ \sum_{(m,n) \in V_b} \ \frac{d_b^2}{N^2} \partial_{ij}^{mn} f(A) \label{eq:firstordertermleading} \\
&+ \sum_{(i,j) \in V_b} \ \sum_{(m,n) \in V_b} \ \partial_{ij}^{mn} f(A) \frac{d_b}{N} \left[ \left( A_{ij} - \frac{d_b}{N} \right) + \left( A_{mn} - \frac{d_b}{N} \right) \right] \label{eq:firstorderleadingnonvanishing} \\
&+ \ N^2 R_3,
\end{align}
where the error term $R_3$ is given by the following equation:
\begin{align}
\frac{1}{4Nd_w} R_3 \ = \ O\left( \frac{N}{d_w} \right) \frac{1}{N^4} \sum_{(i,j),(m,n) \in V_b} \ \partial_{ij}^{mn} f(A) \left( A_{ij} - \frac{d_b}{N} \right) \left( A_{mn} - \frac{d_b}{N} \right).
\end{align}
We first study the term on the RHS of \eqref{eq:firstordertermleading}. We note that, upon averaging over all $(i,j) \in V_b$ and $(m, n) \in V_b$, we have
\begin{align}
\sum_{(i,j) \in V_b} \ \sum_{(m,n) \in V_b} \ \partial_{ij}^{mn} f(A) \ = \ 0,
\end{align}
upon unfolding the definition of the operator $\partial_{ij}^{mn}$ in terms of partial derivatives along directions in $\Delta_{ij} \in \mathscr{X}$. Thus, the term on the RHS of \eqref{eq:firstordertermleading}, i.e. the supposed leading-order term of the expansion of the first-order differential term, in fact vanishes. We now study the term \eqref{eq:firstorderleadingnonvanishing}; upon switching indices $(i,j) \leftrightarrow (m,n) \in V_b$, because we average over indices $(i,j) \in V_b$ and $(m,n) \in V_b$ we see this term is given by
\begin{align}
2 \sum_{(i,j) \in V_b} \ \sum_{(m,n) \in V_b} \ \partial_{ij}^{mn} f(A) \frac{d_b}{N} \left( A_{ij} - \frac{d_b}{N} \right) \ = \ 2 d_b^{1/2} \ \sum_{(i,j) \in V_b} \ \sum_{(m,n) \in V_b} \ \frac{d_b}{N} X_{ij} \partial_{ij}^{mn} f(A),
\end{align}
where the equation follows from the definition of the normalization $X$ of $A$. This, in turn, is equal to 
\begin{align}
\frac{d_b}{2N} \sum_{(i,j), (m,n) \in V_b} \ X_{ij}^{mn} \partial_{ij}^{mn} F(H)
\end{align}
where the factor of $d_b^{-1/2}$ comes from the convention $f(A) = F(X)$, and the replacement of $X_{ij}$ by $X_{ij}^{mn}$ comes with the normalization factor of $1/4$. Thus, we ultimately see, recalling $M = \alpha N$ and $d_w = \alpha d_b$,
\begin{align}
\frac{1}{4 N d_w} \sum_{(i,j), (m,n) \in V_b} \ A_{ij} A_{mn} \partial_{ij}^{mn} f(A) \ = \ \frac{1}{16MN} \sum_{(i,j), (m,n) \in V_b} \ X_{ij}^{mn} \partial_{ij}^{mn} F(H) \ + \ N^2 R_3.
\end{align}
We now study the second-order terms in \eqref{eq:taylorapproximationgenerator}. Again using the representations \eqref{eq:efexpansionij} and \eqref{eq:efexpansionmn}, this term is given by
\begin{align}
\frac{1}{4 N d_w} \sum_{(i,j), (m,n) \in V_b} \ A_{ij} A_{mn} \times \frac12 \left( \partial_{ij}^{mn} \right)^2 f(A) \ = \ &\frac{d_b}{8 d_w N^3} \sum_{(i,j),(m,n) \in V_b} \ \left( \partial_{ij}^{mn} \right)^2 f(A) \\
&+ \ N^2(R_4 + R_5),
\end{align}
where the error terms are given by
\begin{align}
R_4 \ &= \ O \left( \frac{N}{d_w} \right) \frac{1}{N^4} \sum_{(i,j),(m,n) \in V_b} \ \left( \left(\partial_{ij}^{mn} \right)^2 f(A) \frac{d_b}{N} \left( A_{ij} - \frac{d_b}{N} \right) \right), \\
R_5 \ &= \ O \left( \frac{N}{d_w} \right) \frac{1}{N^4} \sum_{(i,j),(m,n) \in V_b} \ \left( \left(\partial_{ij}^{mn} \right)^2 f(A) \left( A_{ij} - \frac{d_b}{N} \right) \left( A_{mn} - \frac{d_b}{N} \right) \right).
\end{align}
Compiling these expansions for the first- and second-order differential terms in \eqref{eq:taylorapproximationgenerator}, we finally deduce
\begin{align}
\mathcal{Q} f(A) \ = \ &\frac{1}{8 MN^2} \sum_{(i,j),(m,n) \in V_b} \ \left( \partial_{ij}^{mn} \right)^2 F(H) \ - \ \frac{1}{16 MN} \sum_{(i,j),(m,n) \in V_b} \ H_{ij} \partial_{ij}^{mn} F(H) \\
&+ \ N^2(R_1 + R_2 + R_3 + R_4 + R_5).
\end{align}
Thus, to finish the proof of Proposition \ref{prop:comparegeneratorsdeterministic}, it suffices to prove the following estimate on the expectation of the error terms $R_1, \ldots, R_5$. 
\begin{prop} \label{prop:errorestimatessobolev}
In the setting of Proposition \ref{prop:comparegeneratorsdeterministic}, we have
\begin{align}
N^2 \E \left[ \sum_{i = 1}^5 \ R_i \right] \ = \ O \left( D^{-1/2} N^{1 + \varepsilon} \right) \max_{1 \leq i \leq 4} \ \| \partial^i F \|_{r(\varepsilon),0}.
\end{align}
\end{prop}
The proof of Proposition \ref{prop:errorestimatessobolev} follows from moment estimates on polynomials in the adjacency matrix entries, which follow from arguments in \cite{BHKY}. For this reason, we delegate these moment estimates to the appendix.
%
%
%
\section{Short-Time Stability of spectral statistics}
\subsection{Aside on Stochastic Inequalities}
We now introduce the following two ubiquitous notions of stochastic inequalities. The second is more important for our purposes, so we emphasize its utility now. 
\begin{definition}
Suppose $\Omega$ is a probability space, and let $\Xi \subseteq \Omega$ be an event. 
\begin{itemize}
\item We say $\Xi$ holds with \emph{high probability} if for every $\zeta > 0$, there exists $N_0(\zeta) > 0$ such that for all $N \geq N_0(\zeta)$, we have $\mathbb{P}( \Xi^C) \leq N^{-\zeta}$.
\item Suppose $A, B$ are two nonnegative random variables. We say that $B$ \emph{stochastically dominates} the random variable $A$ if for any $\zeta > 0$, there exists a $N_0(\zeta) > 0$ such that for all $N \geq N_0(\zeta)$, we have
\begin{align}
\mathbb{P} \left( A > N^{1/\zeta} B \right) \ \leq \ N^{-\zeta}.
\end{align}
In this case, we will adopt the notation $A \prec B$ or $A = O_{\prec} B = O_{\prec}(B)$. 
\end{itemize}
\end{definition}
We conclude this short discussion on stochastic inequalities with the following lemma whose proof is a straightforward application of the definition of stochastic domination.
\begin{lemma}
Suppose $I$ is an indexing set of size $|I| \leq N^{O(1)}$, and suppose $\{ A_i \}_{i \in I}, \{ B_i\}_{i \in I}$ are two families of nonnegative random variables such that for each $i \in I$, we have $A_i \prec B_i$. Then for any nonnegative random variables $\{ c_i \}_{i \in I}$, we have
\begin{align}
\sum_{i \in I} \ c_i A_i \ \prec \ \sum_{i \in I} \ c_i B_i.
\end{align}
\end{lemma}
\subsection{Eigenvector delocalization: the moment flow}
Here, we aim to prove a eigenvector delocalization for all times $t \geq 0$ with an a priori delocalization assumption for the initial data $X_{\ast}(0)$. First, we adopt the following notation. 
\begin{notation}
For the adjacency matrix $X= (H, H^\ast)$ of a bipartite graph, consider the covariance matrix $X_\ast = H^\ast H$. For any given eigenvalue $\lambda_\alpha$ of $X_\ast$, we will denote the corresponding eigenvector by $\mathbf{v}_\alpha$. We also let $\mathbf{v}_{t, \alpha}$ denote the eigenvector of $X_\ast(t)$ corresponding to the eigenvalue $\lambda_\alpha(t)$.
\end{notation}
The precise delocalization result is given as follows.
\begin{prop} \label{prop:eigenvectordelocalization}
Let $\mathbf{q} \in V$ be a vector such that for some fixed $B > 0$, we have the a priori estimate
\begin{align}
\max_\alpha \ |\mathbf{q} \cdot \mathbf{v}_{0,\alpha}| \ \leq \ B.
\end{align}
Then, for any $t \geq 0$, we have
\begin{align}
\max_{\alpha} \ \left| \mathbf{q} \cdot \mathbf{v}_{t, \alpha} \right| \ \prec \ B.
\end{align}
\end{prop}
Using the Chebyshev inequality, it suffices to bound the moments of the dot product $\mathbf{q} \cdot \mathbf{v}_{t, \alpha}$. To bound these moments, we define the following pseudo-moment generating function for this dot product for any vector $\eta = (\eta_i)_{i = 1}^N$ as follows:
\begin{align}
f_t(\eta; \mathbf{q}) \ = \ \E_{\mathbf{\lambda}_t} \left[ \prod_{i = 1}^N \ \frac{1}{(2 \eta_i - 1)!!} \left( \mathbf{q} \cdot \mathbf{v}_{t,i} \right)^{2 \eta_i} \right],
\end{align}
Here, the subscript $\lambda_t$ in the expectation denotes an expectation conditioning on the eigenvalue process $\{\lambda_{\alpha}(t)\}_{\alpha}$. We also use the notation $(n+1)!! = (n+1)(n-1)\ldots1$ for any odd integer $n \in \Z$, and use the convention $(-1)! = 1$.

Thus, our goal will be to estimate this pseudo-MGF. We study this pseudo-MGF $f_t(\eta; \mathbf{q})$ in context of a particle random walk on the lattice $[[1, N]]$ defined through a generator. The construction and study of this $p$-particle random walk in the context of eigenvector delocalization is given in more context and detail in the paper \cite{BY}. To describe this random walk, we first define its configuration space $\Omega_p$ for a fixed $p = O(1)$ independent of $N$:
\begin{align}
\Omega = \Omega_p \ := \ \left\{ \eta \in \Z_{\geq 0}^N: \ \sum_{i = 1}^N \ \eta_i = p \right\}.
\end{align}
We now define the random walk process; to do so, we first give the following definition. 
\begin{notation}
For any configuration $\eta \in \Omega_p$ and any two fixed sites $i \neq j \in [[1, N]]$, we define $\eta^{ij} \in \Omega_p$ as the configuration obtained by 
\begin{align}
\eta^{ij}_k \ = \ \begin{cases} 
      \eta_k & k \neq i,j \\
      \eta_i - 1 & k = i, \ \eta(i) \neq 0 \\
      \eta_i & k = i, \ \eta(i) = 0 \\
      \eta_j + 1 & k = j, \ \eta(i) \neq 0 \\
      \eta_j & k = j, \ \eta(i) = 0
   \end{cases}.
\end{align}
\end{notation}
In words, the configuration $\eta^{ij}$ is the configuration obtained from the configuration $\eta$ by moving one particle at site $i$ to site $j$; if, in the configuration, the site $i$ is void of any particles, then $\eta^{ij} = \eta$, i.e. the configuration $\eta$ is unchanged. This definition now allows us to define a continuous-time $p$-particle jump process on $[[1, N]]$ through the following generator:
\begin{align}
\mathscr{L}_t (f) \ = \ \sum_{i \neq j} \ d_{ij}(t) 2 \eta_i (1 + 2 \eta_j) \left[ f(\eta^{ij}) - f(\eta) \right], \label{eq:differentialoperatorEMF}
\end{align}
where the weight $d_{ij}(t) 2\eta_i(1 + 2 \eta_j)$ determines the particle jump rate. In our situation, we are interested in the weights
\begin{align}
d_{ij}(t) \ = \ \frac{\lambda_{i}(t) + \lambda_{j}(t)}{N(\lambda_{i}(t) - \lambda_{j}(t))^2}. \label{eq:spectralweightemf}
\end{align}
Here, the eigenvalues $\lambda_i(t), \lambda_j(t)$ are the eigenvalues of the time-evolved covariance matrix $X_\ast(t)$. With these weights, we now give the corresponding Kolmogorov backward equation:
\begin{align}
\partial_t f_t(\eta) \ = \ \mathscr{L}_t f_t(\eta). \label{eq:kfe}
\end{align}
This ODE is well-posed because $\Omega$ is finite. Moreover, its solution for some initial condition $f_0$ is given by the pseudo-MGF $f_t(\eta; \mathbf{q})$ so long as the spectrum of the initial data $X(0)$ is simple. A proof of this is given by a direct calculation using the following SDE for the eigenvector dynamics:
\begin{align}
\d\mathbf{v}_{t, \alpha} \ &= \ \frac{1}{\sqrt{N}} \sum_{\beta \neq \alpha} \ \frac{\sqrt{\lambda_\alpha + \lambda_\beta}}{\lambda_\alpha - \lambda_\beta} \mathbf{v}_{t, \beta} \d B_{\alpha \beta}^{(s)} \ - \ \frac{1}{2N} \sum_{\beta \neq \alpha} \ \frac{\lambda_\alpha + \lambda_\beta}{(\lambda_\alpha - \lambda_\beta)^2} \mathbf{v}_{t, \alpha} \d t. \label{eq:DBMeigenvector}
\end{align}
For details on this SDE, see \cite{BY}.

We now exploit study the generator $\mathscr{L}_t$ of this $p$-particle random walk by exploiting its contraction property on any $\ell^{r}(\Omega)$ space. This follows from an application of the Duhamel formula and semigroup theory. 
\begin{proof}
(of Proposition \ref{prop:eigenvectordelocalization}).

Suppose first that the spectrum of $X(0)$ is simple, in which case the pseudo-MGF $f_t(\eta; \mathbf{q})$ solves the backward equation \eqref{eq:kfe}. Because $\mathscr{L}_t$ is a contraction on $\ell^{\infty}$, we see
\begin{align}
\| f_t \|_{\ell^{\infty}(\Omega_p)} \ \leq \ \| f_0 \|_{\ell^{\infty}(\Omega_p)} \ \leq \ B^{2p},
\end{align}
where for the second inequality, we used delocalization for the initial data and the assumption $\eta \in \Omega_p$ to obtain $\| \eta \|_{\ell^1(\Z)} = p$. For any eigenvector index $i \in [[1, N]]$, we pick the configuration $\eta(i)$ whose components are given by $\eta_j = p \delta_{ij}$ which gives
\begin{align}
f_t(\eta(i)) \ = \ \frac{1}{(2p - 1)!!} \E \left( \mathbf{q} \cdot \mathbf{v}_{i,t} \right) \ \leq \ B^{2p}.
\end{align}
Using the Chebyshev inequality, we deduce Proposition \ref{prop:eigenvectordelocalization} in the case that the spectrum of $X(0)$ is simple. In the case that the spectrum of $X(0)$ is not simple, we note that the eigenvectors $\mathbf{v}_{i,t}$ are uniformly continuous in the eigenvalues $\lambda_i(t)$, and thus a perturbation in the spectrum of $X(0)$ reduces the problem to the calculation for the case of $X(0)$ retaining a simple spectrum. This completes the proof of delocalization.

\end{proof}

\subsection{The Free Convolution}
We now introduce the free convolution measure, which interpolates between the spectral statistics of the graph $X(0)$ and the Gaussian perturbation $W$. To define the measure, we first define a Stieltjes transform through the following functional equation:
\begin{align}
m_{\on{lin},t}(z) \ = \ \frac{1}{2N} \sum_{i = 1}^N \ \frac{1}{V_i - z - T m_{\on{lin},t}(z)} + \frac{1}{- V_i - z - Tm_{\on{lin},t}(z)}.
\end{align}
Here, the terms $\pm V_i$ correspond to the nonzero eigenvalues of the linearization $X(t)$, i.e. those eigenvalues $\lambda$ such that $\lambda^2$ is an eigenvalue of the covariance matrix $X_\ast(t)$. It is known that the above fixed-point equation admits a unique solution $m_{\on{lin},t}(z)$ that maps the upper-half plane to itself; we refer to \cite{Bi} for details. 

We now define the Stieltjes transform of the free convolution measure, which we denote by $m_{\on{fc},t}$, by the following equation:
\begin{align}
m_{\on{fc},t}(z) \ = \ \frac{1}{\sqrt{z}} m_{\on{lin}, t}(\sqrt{z}),
\end{align}
where the square root is chosen with the principal branch of the logarithm on the upper-half plane.

The Radon-Nikodym density $\varrho_{\on{fc}}$ of the free convolution measure with respect to Lebesgue measure is defined by taking the Stieltjes inversion of $m_{\on{fc},t}(z)$. We now record a central result in \cite{Y3} concerning a strong local law for the free convolution. First, we will establish the following notation for the partial Stieltjes transform of the linearization $X(t)$:
\begin{align}
s_{\on{lin},t}(z) \ := \ \frac{1}{2N} \sum_{i = M+1}^{M+N} \ \left( X(t) - z \right)^{-1}_{ii}.
\end{align}
%
\begin{theorem}
Fix any $\zeta > 0$. Then for any time $t \in [0, N^{-\zeta}]$ we have with high probability, uniformly over the domain $U_{\e, \delta}$ defined shortly, 
\begin{align}
\left| s_{\on{lin},t}(z) - m_{\on{lin},t}(z) \right| \ \prec \ \frac{1}{\sqrt{N \eta}}. \label{eq:shorttimesSTfreeconvolutionlinear}
\end{align}
\end{theorem}
We now use this result to establish a similar weak local law for the Stieltjes transform $s_t(z)$ of the covariance matrix $X_\ast(t)$. We first note the following relation:
\begin{align}
s_t(z) \ = \ \frac{1}{\sqrt{z}} s_{\on{lin}, t}(\sqrt{z}),
\end{align}
where we again take the principal branch of the logarithm in defining the square root on the upper-half plane. Because $z \in U_{\e, \delta}$ implies $z^2 \in U_{\e', \delta'}$ for some $\e', \delta' > 0$, and because $|z|^{-1} = O(1)$ for all $z \in U_{\e, \delta}$, we deduce the following weak local law for the Stieltjes transform of the covariance matrix for any time $t \in [0, N^{-\zeta}]$:
\begin{align}
\sup_{z \in U_{\e, \delta}} \ \left| s_t(z) - m_{\on{fc}, t}(z) \right| \ \prec \ \frac{1}{\sqrt{N \eta}}. \label{eq:shorttimesstieltjesfreeconvolutionMP}
\end{align}
We now use this estimate and Theorem 2.5 in \cite{Y1} to deduce a local law for all short times. First, we define the subdomain of $U_{\e, \delta}$:
\begin{align}
U_{\e, \delta}^1 \ = \ \left\{ z = E + i \eta \in U_{\e, \delta}: \ |\eta| \leq 1 \right\}.
\end{align}
%
\begin{prop} \label{prop:locallawOUlimit}
For some $B \leq N^{-\zeta}$ for any fixed sufficiently small $\zeta > 0$, we have the following local law uniformly over such $z$ and uniformly over all times $t \in [0, B]$:
\begin{align}
\left| s_t(z) - m(z) \right| \ \prec \ B + \frac{1}{(N \eta)^{1/4}}. \label{eq:STOUINFTY}
\end{align}
\end{prop}
\begin{proof}
Using \eqref{eq:shorttimesstieltjesfreeconvolutionMP} and Theorem 2.5 in \cite{Y1} (the local law), we note it suffices to prove, instead, the estimate
\begin{align}
\left| m(z) - m_{\on{fc},t}(z) \right| \ \prec \ B + \frac{1}{(N \eta)^{1/4}}. \label{eq:STFCINFTY}
\end{align}
Moreover, by the same self-improving estimate in Lemma 4.10, we may restrict ourselves to $U_{\e, \delta}^1$. We now unfold the definition of the Stieltjes transform $m_{\on{fc},t}$ to obtain the following fixed-point equation for the Stieltjes transform:
\begin{align}
m_{\on{fc},t}(z) \ = \ \frac{1}{N} \sum_{i = 1}^N \ \frac{1 +  tm_{\on{fc},t}(z)}{V_i^2 - z(1 + tm_{\on{fc},t}(z))^2}
\end{align}
This allows us to rewrite the free convolution $m_{\on{fc},t}$ in terms of the Stieltjes transform $s_0(z)$ for the bipartite graphs as follows:
\begin{align}
m_{\on{fc},t}(z) \ = \ \left( 1 + tm_{\on{fc},t}(z) \right) s_0 ( z(1 + tm_{\on{fc},t}(z))^2 ). \label{eq:initialdatarepresentationSTfreeconvolution}
\end{align}
By Lemma 7.1 and Lemma 7.2 in \cite{LY}, we first note
\begin{align}
m_{\on{fc},t}(z) \ = \ O \left( \log N \right). \label{eq:STlogbound}
\end{align}
Using the initial data representation \eqref{eq:initialdatarepresentationSTfreeconvolution} and the self-consistent equation for the bipartite graph data $s_0$ given in Proposition 5.3 in \cite{Y1}, we have the following self-consistent equation for the free convolution data $m_{\on{fc},t}$:
\begin{align}
\gamma z \left(1 + t m_{\on{fc},t}(z) \right)^2 \frac{m_{\on{fc},t}(z)^2}{(1 + t m_{\on{fc}, t}(z))^2} + \left(\gamma + z(1 + t m_{\on{fc},t}(z))^2 - 1 \right) \frac{m_{\on{fc},t}(z)}{1 + t m_{\on{fc},t}(z))} \ - \ 1 \ = \ (1 + |z|) o_{N \to \infty}(1)
\end{align}
We rewrite the first term as $\gamma z m_{\on{fc},t}(z)^2$. We now focus on the second term. Expanding it, we have the second term is given by the following expression:
\begin{align}
m_{\on{fc},t}(z) \left[ (\gamma -1) \frac{1}{1 + t m_{\on{fc},t}(z)} \ + \ z(1 + tm_{\on{fc},t}(z)) \right] \ &= \ m_{\on{fc},t}(z) \left[ (\gamma + z - 1) \ + \ \mathscr{E}_1 \ + \ \mathscr{E}_2 \right],
\end{align}
where the error terms are given by
\begin{align}
\mathscr{E}_1 \ &= \ (\gamma - 1) \frac{t m_{\on{fc},t}(z)}{1 + t m_{\on{fc},t}(z)} \\
\mathscr{E}_2 \ &= \ t \left( z  m_{\on{fc},t}(z) \right).
\end{align}
Here, we use $\log(N)$-bound in \eqref{eq:STlogbound} to deduce the following deterministic bound:
\begin{align}
\sup_{z \in U_{\e, \delta}} \ \left| tm_{\on{fc},t} \right| = o(1).
\end{align}
This allows us to Taylor expand the term $(1 + tm_{\on{fc},t}(z))^{-1}$ via a geometric series. Ultimately, this expansion of the second term gives us the following self-consistent equation
\begin{align}
\gamma z m_{\on{fc},t}^2 \ + \ (\gamma + z - 1) m_{\on{fc},t} - 1 \ = \ (1 + |z|) o_{N \to \infty}(1) \ - \ m_{\on{fc},t} (\mathscr{E}_1 + \mathscr{E}_2).
\end{align}
Because $t \ll N^{-\zeta}$ for sufficiently small $\zeta > 0$, again by \eqref{eq:STlogbound} we deduce the following bound:
\begin{align}
\sup_{z \in U_{\e, \delta}} \ \left| m_{\on{fc},t}(z)(\mathscr{E}_1 + \mathscr{E}_2) \right| \ = \ (1 + |z|) o_{N \to \infty}(1).
\end{align}
We now obtain the estimate \eqref{eq:STFCINFTY} by appealing to Proposition 5.4 in \cite{Y1}. This gives the desired estimate for a fixed time $t \in [0, B]$. To extend to all times $t \in [0, B]$, we appeal to a stochastic continuity argument for the Stieltjes transform given \cite{Y3}. This completes the proof of Proposition \ref{prop:locallawOUlimit}.
\end{proof}
We now briefly discuss the consequences of Proposition \ref{prop:locallawOUlimit}. We first recall the following notion: for any index $i \in [[1, N]]$, we define the $i$-th classical location $\gamma_i$ of the Marchenko-Pastur law as the solution to the following quantile formula:
\begin{align}
\frac{i}{N} \ = \ \int_{-\infty}^{\gamma_i} \ \varrho_{\on{MP}}(E) \ \d E,
\end{align}
where $\varrho_{\on{MP}}$ denotes the density function of the Marchenko-Pastur law. We will also introduce the following maximal functions for the Green's function borrowed from \cite{Y1}.
\begin{notation}
Let $G_t(z)$ denote the Green's function of $X_\ast(t)$, i.e.
\begin{align}
G_t(z) \ = \ \left( X_\ast(t) - z \right)^{-1}
\end{align}
for $z \in \C_+$. We will define the following control parameter for the Green's function:
\begin{align}
\Gamma_t(z) \ = \ \max_{i, j \in [[1, N]]} \ \left( \left| G_{ij}(z) \right| \vee 1 \right).
\end{align}
\end{notation}
We now record the following estimates for eigenvalue rigidity and the maximal function $\Gamma_t$, which are consequence of Proposition \ref{prop:locallawOUlimit}.
\begin{prop} \label{prop:consequenceslocallawOUlimit}
For a fixed $\kappa > 0$ independent of $N$, fix an index $i \in [[\kappa, (1-\kappa)N]]$. Then 
\begin{align}
\sup_{t \in [0, D^{-1/4}]} \ \left| \lambda_i(t) - \gamma(i) \right| \ \prec \ D^{-1/4} \label{eq:weakrigidityuniformtime}
\end{align}
where $\lambda_i(t)$ denotes the $i$-th eigenvalue of $X_\ast(t)$ and $\gamma(i)$ denotes the $i$-th classical location of $\varrho_{\on{MP}}$. Moreover, suppose we have an interval $\mathscr{I} \subseteq [\kappa, (1 - \kappa)(1 + \sqrt{\gamma})^2]$ such that $|\mathscr{I}| \asymp N^{-1 + \zeta}$ for some fixed small $\zeta > 0$. Then for any time $t \in [0, D^{-1/4}]$,
\begin{align}
\int_{\mathscr{I}} \ \sum_{i = 1}^N \ \delta_{\lambda_i(t)} \ = \ | \left\{ i: \ \lambda_i(t) \in \mathscr{I} \right\} | \ = \ O(N^{\zeta}) \ = \ O_{\prec} (1). \label{eq:accumulationeigenvalues}
\end{align}
Lastly, for any $z \in \C_+$ we have
\begin{align}
\sup_{t \in [0, D^{-1/4}]} \ \Gamma_t(z) \ \prec \ 1 + \frac{1}{N \eta}. \label{eq:GFsupnorm}
\end{align}
\end{prop}
\begin{proof}
The proof of the weak rigidity estimate \eqref{eq:weakrigidityuniformtime} follows from the Helffer-Sjostrand functional calculus. The statement of accumulating eigenvalues \eqref{eq:accumulationeigenvalues} follows from \eqref{eq:weakrigidityuniformtime} and regularity of the Marchenko-Pastur density inside its bulk.

It now remains to prove the second estimate \eqref{eq:GFsupnorm} which follows from a dyadic decomposition of the scale $\eta$ is given in Proposition 5.1 in \cite{BHKY} and an eigenvector delocalization result proven in the next section. We briefly and loosely discuss a variant of this argument using the estimate \eqref{eq:accumulationeigenvalues}. We first appeal to the spectral representation of the Green's function $G_t(z)$ as follows: for any fixed indices $i,j$, we have
\begin{align}
\left| [G_t(z)]_{ij} \right| \ = \ \left| \sum_{k = 1}^N \ \frac{\mathbf{v}_k(i) \mathbf{v}_k(j)}{\lambda_k(t) - z} \right| \ &\leq \ \left( \sup_{k, i, j \in [[1, N]]} \ |\mathbf{v}_k(i) \mathbf{v}_k(j)| \right) \sum_{k = 1}^N \ \frac{1}{|\lambda_k(t) - z|} \\
&\prec \ \frac{1}{N} \sum_{k = 1}^N \ \frac{1}{|\lambda_k(t) - z|}.\label{eq:termtoestimateGFsup}
\end{align}
where to obtain the last inequality we appeal to eigenvector delocalization in Proposition \ref{prop:eigenvectordelocalization}. Instead of the dyadic decomposition used in \cite{BHKY}, we now differentiate in $\eta$:
\begin{align}
\partial_{\eta} \ \frac{1}{\sqrt{(\lambda_k(t) - E)^2 + \eta^2}} \ = \ - \frac{\eta}{\left[ (\lambda_{\alpha} - E)^2 + \eta^2 \right]^{3/2}}.
\end{align}
Thus by the time-evolving local law in Proposition \ref{prop:locallawOUlimit} and the bound \eqref{eq:termtoestimateGFsup} we see, writing $z = E + i \eta_0$, for $\e > 0$ sufficiently small
\begin{align}
\left| [G_t(z)]_{ij} \right| \ \prec \ 1 + \frac{1}{N} \sum_{k = 1}^N \ \int_{\eta_0}^{N^{-1 + \e}} \ \frac{\eta}{\left[ (\lambda_{\alpha} - E)^2 + \eta^2 \right]^{3/2}} \ \d \eta. \label{eq:integralestimateGFsup}
\end{align}
It is now our goal to bound the integral on the RHS of \eqref{eq:integralestimateGFsup}. To do so, we first define the following sets of eigenvalues $\lambda_k(t)$ of $X_\ast(t)$: for a fixed, small $\zeta > 0$,
\begin{align}
\mathscr{I}_1 \ &= \ \left\{ k: \ |\lambda_k - E| \leq N^{-1 + \zeta} \right\}, \quad \mathscr{I}_{2,\ell} \ = \ \left\{ k: \ 2^{\ell} N^{-1 + \zeta} \leq |\lambda_k - E| < 2^{\ell + 1} N^{-1 + \zeta} \right\}
\end{align}
with $\ell = 0, \ldots, O(\log N)$. We first compute an upper bound on the integral on the RHS of \eqref{eq:integralestimateGFsup} by restricting to those eigenvalues $\lambda_k$ with $k \in \mathscr{I}_{2,\ell}$ for any such $\ell$. In this case, we have the following estimates:
\begin{align}
\int_{\eta_0}^{N^{-1 + \e}} \ \frac{\eta}{\left[ (\lambda_{\alpha} - E)^2 + \eta^2 \right]^{3/2}} \d\eta \ &\leq \ \int_{\eta_0}^{N^{-1 + \e}} \ \frac{\eta}{2^{3 \ell} N^{-3 + 3 \zeta} \left[ 1 + (2^{-\ell} N^{1 - \zeta} \eta)^2 \right]^{3/2}} \d\eta \nonumber \\
&= \ 2^{-\ell} N^{1 - \zeta} \int_{\eta_0}^{N^{-1 + \e}} \ \frac{2^{-\ell} N^{1 - \zeta}\eta}{\left[1 + (2^{-\ell}N^{1 - \zeta} \eta)^2 \right]^{3/2}} \d\left(2^{-\ell} N^{1 - \zeta} \eta \right) \nonumber \\
&= \ 2^{-\ell} N^{1 - \zeta} \int_{2^{-\ell}N^{1 - \zeta} \eta_0}^{2^{-\ell}N^{\e - \zeta}} \ \frac{\underline{\eta}}{\left[ 1 + \underline{\eta}^2 \right]^{3/2}} \d\underline{\eta}, \label{eq:changeofvariables}
\end{align}
where we employed a change of variables $\underline{\eta} = 2^{-\ell} N^{1 - \zeta} \eta$. We note the integrand in \eqref{eq:changeofvariables} is increasing in $\eta$ for an interval $[0, L]$ with $L = O(1)$. Thus, up to an $O(1)$ term the integral in \eqref{eq:changeofvariables} is bounded above by
\begin{align}
O \left( \frac{L}{(1 + L^2)^{3/2}} \right) \ = \ O(1)
\end{align}
for $N \gg 1$ if we choose $\e < \zeta$. By \eqref{eq:accumulationeigenvalues}, this implies
\begin{align}
N^{-1} \sum_{k \in \mathscr{I}_{2,\ell}} \ \int_{\eta_0}^1 \ \frac{\eta}{\left[ (\lambda_{\alpha} - E)^2 + \eta^2 \right]^{3/2}} \d\eta \ \prec \ N^{-1} \sum_{k \in \mathscr{I}_{2,\ell}} 2^{-\ell} N^{1 - \zeta} \ \prec \ 1.
\end{align}
Because $\ell = 0, \ldots, O(\log N)$, the contribution from eigenvalues $\lambda_k$ that are distance between $N^{-1 + \zeta}$ and $2$ from the energy $E$ is $O_{\prec}(\log N) = O_{\prec}(1)$.

We now study the eigenvalues $\lambda_k$ with the index $k \in \mathscr{I}_1$. To control this term, we simply compute the following estimate:
\begin{align}
\frac{1}{N} \sum_{k \in \mathscr{I}_1} \ \int_{\eta_0}^1 \ \frac{\eta}{\left[ |\lambda_k - E|^2 + \eta^2 \right]^{3/2}}\d\eta \ \prec \ \frac{1}{N} \int_{\eta_0}^1 \ \frac{\eta}{\eta^3} \d\eta \ \leq \ \frac{1}{N} + \frac{1}{N \eta}, 
\end{align}
where we used \eqref{eq:accumulationeigenvalues} to estimate the size $|\mathscr{I}_1| \prec 1$. For eigenvalues $\lambda_k \not\in \mathscr{I}_1, \mathscr{I}_{2,\ell}$ for any $\ell = 0, \ldots, O(\log N)$, we know $|\lambda_k - E| > 2$, in which case we have the trivial estimate
\begin{align}
\frac{1}{N} \sum_{\lambda_k \not\in \mathscr{I}_1, \mathscr{I}_{2,\ell}} \ \frac{1}{\left| \lambda_k(t) - z \right|} \ \prec \ \frac{1}{|\lambda_k(t) - E|} \ \prec \ 1.
\end{align}
This completes the proof of Proposition \ref{prop:consequenceslocallawOUlimit}.
\end{proof}
\begin{remark}
Instead of partitioning the spectrum of $X_{\ast}(t)$ into the sets $\mathscr{I}_1$ and $\mathscr{I}_2$, we may also appeal to a weak level repulsion estimate proved in \cite{Y3}.
\end{remark}
%
%
%
\section{Proof of Theorem \ref{theorem:averagedcorrelation}}
\subsection{Preliminary Estimates}
We first record two important estimates to control Green's functions of perturbed linearized covariance matrices.
\begin{lemma} \label{lemma:perturbedbounds}
Fix a positive integer $n \geq 0$. For any spectral parameter $z = E + i \eta \not\in \R$ with $|E| > \e$ for a fixed $\e > 0$, and a short time $t \leq D^{-1/4}$, we have
\begin{align}
\sup_{\theta \in [0,1]^n} \sup_{Y \in \mathscr{X}^n} \ \Gamma \left( z; X(t) + d^{-1/2} \theta \cdot Y \right) \ \prec \ 1 + \frac{1}{N \eta}. \label{eq:perturbedGFsupbound}
\end{align}
Here, the function $\Gamma(z; H)$ denotes the entry-wise maximum of the Green's function $G(z; H) = (H - z)^{-1}$ of the matrix $H$ as used before. Here, we allow either $d = d_b$ or $d = d_w$, without a change to the estimate.
\end{lemma}
\begin{proof}
Before we proceed with any calculations, we reduce the problem to the following three assumptions.
\begin{itemize}
\item First, by the following relation which holds for any real matrix $H$:
\begin{align}
G(\overline{z}; H) \ = \ \overline{G(z; H)} 
\end{align}
where the RHS is the entry-wise complex-conjugate of the Green's function $G(z; H)$, we may assume $\eta > 0$.
\item Second, we may assume $|E| \prec 1$ by the high-probability $O(1)$ bound on eigenvalues of normalized adjacency matrices of biregular graphs (see \cite{DJ}) and the following  perturbation inequality for eigenvalues: 
\begin{align}
\lambda(V-W) \leq \| V - W \|_{\infty}, \nonumber
\end{align}
where $V, W$ are real symmetric matrices. Similarly, we may assume $\eta \prec 1$. Thus, we may assume $|z| \asymp 1$.
\item Lastly, by the same bootstrapping method used in the proof of the local law in Lemma 4.10 in \cite{Y1}, we may work in the regime $\eta \gg N^{-1}$. 
\end{itemize}
Proceeding with the derivation of \eqref{eq:perturbedGFsupbound}, we first compute the following spatial derivatives of $G(z; X_\ast(t))$ via the resolvent formula:
\begin{align}
\partial_{X_1 \ldots X_n} \ G \ = \ (-1)^n \sum_{\sigma \in S_n} \ G X_{\sigma(1)} G \ldots X_{\sigma(n)} G. \label{eq:derivativeformulaGF}
\end{align}
Here, $S_n$ denotes the permutation group of $n$ letters. With this, we write down the following Taylor estimate:
\begin{align}
\left| G_{ij}(z; X(t) + d^{-1/2} \theta \cdot Y) \right| \ &\leq \ \left| G_{ij}(z; X(t)) \right| + \sup_{\theta \in [0,1]^n} \sup_{X \in \mathscr{X}^n} \ \left| \nabla_X G \cdot \theta \right| \nonumber \\
&\prec \ |z| \left( 1 + \frac{1}{N \eta} \right) + O_n\left( \sup_{X \in \mathscr{X}} \partial_X G \right).
\end{align}
To bound the error term above, we give the following estimate:
\begin{align}
\left( G Y G \right)_{ij} \ = \ \sum_{k, \ell} \ G_{ik} Y_{k \ell} G_{\ell j} \ \leq \ O \left( \Gamma(z; X_{\ast}(t)) \right)^2 \ \prec \ |z|^2 \left(1 + \frac{1}{N \eta} \right)^2 \ \prec \ 1 + \frac{1}{N \eta} + \frac{1}{(N \eta)^2}.
\end{align}
We note the implied constant in the big-Oh term may be chosen independent of $i,j$ as any matrix $X \in \mathscr{X}$ has finitely many non-zero terms, which are all bounded. Given the assumption $\eta \succ N^{-1}$, we know
\begin{align}
\frac{1}{(N \eta)^2} \ \prec \ \frac{1}{N \eta}.
\end{align}
Because $|z| \asymp 1$, we also know $|z|^2 = O(|z|)$. With this, we see
\begin{align}
\left| G_{ij}(z; X(t) + d^{-1/2} \theta \cdot Y) \right| \ \prec \ 1 + \frac{1}{N \eta} + 1 + \frac{1}{N \eta} + \frac{1}{(N \eta)^2} \ \prec \ |z| \left( 1 + \frac{1}{N \eta} \right),
\end{align}
which completes the derivation of \eqref{eq:perturbedGFsupbound}.
\end{proof}
Before we proceed with the proof of Theorem 4.4, we record the following consequences, which give the same Green's function estimate \eqref{eq:perturbedGFsupbound} for covariance matrices. Before we do so we introduce the following notation that will only be used in stating the consequences of Lemma \ref{lemma:perturbedbounds}.
\begin{notation}
Suppose $X(t) = (H(t), H(t)^\ast)$, so that $X_{\ast}(t) = H(t)^\ast H(t)$. We establish the following notation for the perturbed covariance matrices: for $\theta \in [0,1]^n$ and $Y \in \mathscr{X}^n$, we define
\begin{align}
X_{\ast}(t; \theta, Y) \ &:= \ \left( H(t) + d^{-1/2} \theta \cdot Y \right)^\ast \left( H(t) + d^{-1/2} \theta \cdot Y \right), \\
X^{\ast}(t; \theta, Y) \ &:= \ \left( H(t) + d^{-1/2} \theta \cdot Y \right) \left( H(t) + d^{-1/2} \theta \cdot Y \right)^{\ast}.
\end{align}
\end{notation}
We record the following consequence for possible future use as it will not be used in this thesis. The proof of the corollary follows immediately from Lemma \ref{lemma:perturbedbounds}, the spectral correspondence between the linearization $X(t)$ and the corresponding covariance matrices, and lastly the spectral representation of the Green's function.
\begin{corollary}
Assuming the setting of Lemma \ref{lemma:perturbedbounds}, the estimate \eqref{eq:perturbedGFsupbound} holds upon replacing $X(t) + d^{-1/2} \theta \cdot Y$ with $X_{\ast}(t; \theta, X)$ and $X^{\ast}(t; \theta, X)$.
\end{corollary}
\subsection{Proof of Theorem \ref{theorem:averagedcorrelation}}
We are now in a position to prove Theorem 4.4. To do so, we rely on the following result concerned with universality of averaged correlation functions in the bulk. The result is taken as Lemma 5.4 in Section 5 in \cite{BHKY}; for a proof, we refer to Theorem 6.4 in \cite{EYY}.
\begin{theorem} \label{theorem:tracecorrelationtoaveragedcorrelation}
Suppose $H_1$ and $H_2$ are two random matrix ensembles of equal dimension $N$, and denote their Green's functions by $G_1(z)$ and $G_2(z)$, respectively. 

Fix a positive integer $n > 0$, and fix a sequence of positive integers $k_1, \ldots, k_n$. Fix a (small) constant $\beta > 0$. For a scale $\eta \in [N^{-1 - \beta}, N^{-1}]$, we fix a sequence of complex numbers $z_j^m = E_j^m \pm i \eta$ for $j \in [[1, k_m]]$ and $m \in [[1, n]]$. Here, we stipulate the energies $E_j^m \in \mathscr{I}_{\on{linear},\e}$ are in the bulk of the Marchenko-Pastur law. Moreover, the signs in the imaginary part of $z_j^m$ are arbitrary.

Let $\varphi \in C_c^{\infty}(\R^n)$ be a smooth function with compact support such that for any multi-index $\nu = (\nu_1, \ldots, \nu_n)$ with $1 \leq |\nu| \leq 4$, the following gradients estimates hold for any $\omega > 0$ fixed and sufficiently small:
\begin{align}
\sup_{x \in [-N^{\omega}, N^{\omega}]} \ \left| \partial^\nu \varphi(x) \right| \ &\leq \ N^{O(\omega)}, \label{eq:gradientestimate1} \\
\sup_{x \in [-N^2, N^2]} \ \left| \partial^{\nu} \varphi(x) \right| \ &\leq \ N^{O(1)}. \label{eq:gradientestimate2}
\end{align}
Lastly, suppose the following estimate holds:
\begin{align}
\left| \E \varphi \left( N^{-k_1} \on{Tr} \left( \prod_{j = 1}^{k_1} \ G_1(z_j^1) \right), \ldots, N^{-k_n} \on{Tr} \left( \prod_{j = 1}^{k_n} \ G_1(z_j^n) \right) \right) \ - \ \E \varphi(G_1 \rightarrow G_2) \right| \ = \ O \left( N^{-\delta/2 + O(\beta)} \right), \label{eq:tracecorrelation}
\end{align}
where the notation $G_1 \rightarrow G_2$ denotes switching all terms depending on $G_1$ to the corresponding terms depending on $G_2$, and the implicit constant is allowed to depend on all data in the statement of this theorem except the dimension $N$. Then, the averaged bulk eigenvalue correlation statistics of $H_1$ and $H_2$ agree in the sense of Definition 4.1.
\end{theorem}
In particular, to prove Theorem 4.4, it will suffice to show that the estimate \eqref{eq:tracecorrelation} holds in our matrix ensembles $\mathscr{X}(0)$ and $\mathscr{X}(t)$ for short times; this is the content of the following proposition.
\begin{prop} \label{prop:derivetracecorrelation}
Fix a small constant $\zeta > 0$, and suppose $t \in [0, N^{-1 - \zeta} D^{1/2}]$. Assuming the setting of Theorem \ref{theorem:tracecorrelationtoaveragedcorrelation} up until the gradient estimates \eqref{eq:gradientestimate1} and \eqref{eq:gradientestimate2} with the following random matrix ensembles of linearized covariance matrices:
\begin{align}
H_1 \ &= \ \mathscr{X}(0), \quad H_2 \ = \ \mathscr{X}(t).
\end{align}
Then the estimate \eqref{eq:tracecorrelation} holds.
\end{prop}
\begin{proof}
We begin by defining the following function:
\begin{align}
F(X(t)) \ = \ \varphi \left( N^{-k_1} \on{Tr} \left( \prod_{j = 1}^{k_1} \ G_1(z_j^1) \right), \ldots, N^{-k_n} \on{Tr} \left( \prod_{j = 1}^{k_n} \ G_1(z_j^n) \right) \right).
\end{align}
This will be treated as a function of $X(t)$ for times $t = 0$ and another time $t \leq N^{-1 - \zeta} D^{1/2}$. By Theorem 5.2 we have the following short-time stability for the expectation of $F(X(t))$:
\begin{align}
\E F(X(t)) - \E F(X(0)) \ = \ O \left( D^{-1/2} N^{1 + \zeta} \max_{1 \leq i \leq 4} \ \int_0^t \ \| \partial^i F \|_{r,s} \ \text{d} s \right).
\end{align}
Thus, it will suffice to prove the following gradient estimate on the function $F$ for times $t \leq N^{-1 - \zeta} D^{1/2}$:
\begin{align}
\max_{1 \leq i \leq 4} \ \| \partial^i F \|_{r,s} \ = \ O \left( N^{\zeta/2 + O(\beta)}  \right)\label{eq:gradientestimateshorttimestability}
\end{align}
and then choose $\zeta, \beta > 0$ sufficiently small. For simplicity and clarity of presentation, we will focus on the case $n = 1$ and $k_1 = 1$; the argument for general $n > 0$ follows similarly.

We now differentiate the function $F(X)$ from definition. In what follows, the Green's function $G(z)$ will denote the Green's function of a perturbed linearized covariance matrix as in Lemma \ref{lemma:perturbedbounds}. In particular, by the chain-rule we have
\begin{align}
\partial_{X_1 \ldots X_i} F \left(X(t) + d^{-1/2} \theta \cdot X \right) \ &= \ \partial_{X_1 \ldots X_i} \varphi \left( \frac{1}{N} \on{Tr} G(z) \right) \\
&= \ \partial_{X_2 \ldots X_i} \left[ \partial_{X_1} \varphi \left( \frac{1}{N} \on{Tr} G(z) \right) \times \frac{1}{N} \on{Tr} \left( \partial_{X_1} G \right) \right] \\
&= \ O_i \left( \max_{1 \leq k \leq i} \ | \varphi^{(k)} | \times \frac{1}{N} \max_{1 \leq k \leq i} \ \left| \on{Tr} \left(\partial_{X_{j_1} \ldots X_{j_k}} G(z) \right) \right| \right).\label{eq:testfunctiongradientbound}
\end{align}
Here, \eqref{eq:testfunctiongradientbound} follows from a repeated application of the Leibniz rule and chain rule for differentiating along switching matrices $X_j \in \mathscr{X}$. We now use the a priori gradient estimate \eqref{eq:gradientestimate1} to bound the first term inside the big-Oh term in \eqref{eq:testfunctiongradientbound} to deduce
\begin{align}
\partial_{X_1 \ldots X_i} F \left((X(t) + d^{-1/2} \theta \cdot X \right) \ = \ O_i \left( \frac{N^{O(\omega)}}{N} \max_{1 \leq k \leq i} \left| \on{Tr} \left( \partial_{X_{j_1} \ldots X_{j_k}} G(z; X(t) + d^{-1/2} \theta \cdot X) \right) \right| \right).
\end{align}
We now bound the trace term appealing back to the differentiation identity \eqref{eq:derivativeformulaGF} which we rewrite as follows:
\begin{align}
\partial_{X_1 \ldots X_i} G \ = \ (-1)^i \sum_{\sigma \in S_i} \ G X_{\sigma(1)} G \ldots X_{\sigma(i)} G. 
\end{align}
Because $i \in [[1, 4]]$ and each $X_{j}$ has at most $O(1)$ non-vanishing entries, we deduce the following straightforward gradient-trace bound, which will help us control the gradient bound \eqref{eq:testfunctiongradientbound}:
\begin{align}
\frac{1}{N} \on{Tr} \left( \partial_{X_1 \ldots X_i} G \right) \ &= \ O \left( \frac{1}{N} \sum_{j = 1}^N \ i! \left[ \max_{\sigma \in S_i} \ \left| G X_{\sigma(1)} \ldots X_{\sigma(i)} G \right| \right]_{jj} \right) \\
&= \ O_i \left( \Gamma^{O_i(1)} \right).
\end{align}
Combining this estimate with the bounds on perturbed Green's functions \eqref{eq:perturbedGFsupbound} in Lemma \ref{lemma:perturbedbounds}, we deduce the following bound
\begin{align}
\partial_{X_1 \ldots X_i} F \left( X(t) + d^{-1/2} \theta \cdot X \right) \ = \ O_i \left( \left( 1 + \frac{1}{N \eta} \right)^{O(1)} \right) \ \prec \ O \left( N^{O(\beta)} \right),
\end{align}
where we used the assumption $\eta \in [N^{-1-\beta}, 1]$ in the last big-Oh estimate. We note this bound holds only with high-probability as the inequality \eqref{eq:perturbedGFsupbound} in Lemma \ref{lemma:perturbedbounds} is a stochastic inequality. For the low-probability complement event, we will go back to the preliminary estimate \eqref{eq:testfunctiongradientbound} and apply straightforward bounds as follows, instead using \eqref{eq:gradientestimate2} as opposed to \eqref{eq:gradientestimate1}:
\begin{align}
\partial_{X_1 \ldots X_i} F \left( X(t) + d^{-1/2} \theta \cdot X \right) \ &= \ O_i \left( \max_{1 \leq k \leq i} \ |\varphi^{(k)}| \times \frac{1}{N} \max_{1 \leq k \leq i} \ \left| \on{Tr} \left( \partial_{X_{j_1} \ldots X_{j_k}} G(z) \right) \right| \right) \\
&= \ O_i \left( N^{O(1)} \eta^{-C} \right) \\
&= \ O_i \left( N^{O(1)} \right),
\end{align}
where $C = O(1)$ is a positive constant. Thus in taking an expectation in the definition of the $\|-\|_{r,s}$ norm, we have
\begin{align}
\| \partial^i F\|_{r,t} \ = \ O \left( N^{\zeta + O(\beta)} N^{-\zeta/r + O(1)} \right) \ = \ O \left( N^{\zeta/2 + O(\beta)} \right)
\end{align} 
upon taking the exponent $r > 0$ suitably small. Here, we drop the subscript $i$ from the big-Oh term because $i \in [[1, 4]]$ is drawn from a set of size $O(1)$. This completes the proof of Proposition \ref{prop:derivetracecorrelation} and thus the proof of Theorem 4.4.
\end{proof}
%
%
%
\section{Appendix}
\subsection{Moment estimates and Proposition \ref{prop:errorestimatessobolev}}
The necessary ingredients for the proof of Proposition \ref{prop:errorestimatessobolev} are moment bounds on the adjacency matrix entries. The first of these bounds is the following estimate, from which we will derive further moment bounds.
\begin{lemma} \label{lemma:pseudoindependenceadjacencymatrix}
Fix any $p = O(1) > 0$ and consider any $p$ vertices $(i_1, j_1), \ldots, (i_p, j_p) \in V_b$. Then, for any $x \in [1, M+N]$ and $y \in [1, M+N] \setminus \{(i_1, j_1), \ldots, (i_p, j_p)\}$, we have
\begin{align}
\E \left[ A_{i_1 j_1} \ldots A_{i_p j_p} A_{xy} \right] \ = \ O \left( \frac{d_b}{N} \right) \E \left[ A_{i_1 j_1} \ldots A_{i_p j_p} \right].
\end{align}
\end{lemma}
\begin{proof}
First, we may assume $(x,y) \in V_b$ or $(x,y) \in V_w$; otherwise, the estimate follows trivially. We derive the estimate for $(x,y) \in V_b$; the case $(x,y) \in V_w$ follows similarly. 

We now define the following for notational convenience:
\begin{align}
I(p) \ := \ \{(i_1 j_1), \ldots, (i_p, j_p) \}, \ \ \ A(p) \ := \ A_{i_1 j_1} \ldots A_{i_p j_p}.
\end{align}
Thus, we have
\begin{align}
\E \left[ A(p) \right] \ = \ \frac{1}{d_b} \E \left[ A(p) \sum_y \ A_{xy} \right] \ = \ \frac{1}{d_b} \E \left[ A(p) \sum_{y \in I} \ A_{xy} \right] \ + \ \frac{1}{d_b} \E \left[ A(p) \sum_{y \not\in I} \ A_{xy} \right].
\end{align}
We note $|I| = O(1)$ and thus $|\{ y \not\in I \}| = O(N)$. Moreover, noting $0< A_{ij} = O(1)$ for all $i,j$, and also noting the law of $A_{ij}$ under the uniform measure on graphs is invariant under relabeling vertices, we deduce
\begin{align}
\E \left[ A(p) \right] \ = \ O \left( \frac{1}{d_b} \right) \E \left[ A(p) A_{xy} \right] \ + \ O \left( \frac{N}{d_b} \right) \E \left[ A(p) \right],
\end{align}
from which the desired estimate follows clearly.
\end{proof}
\begin{remark}
The estimate given above is somewhat of an independence statement under mild conditions; this will help us compute estimates for moments of adjacency matrix entries. In particular, we deduce the following moment estimates.
\end{remark}
\begin{lemma} \label{lemma:momentestimatesadjacencymatrix}
Let $a,b$ be integers defined by
\begin{align}
\# \left\{ i,j,m,n \right\} \ = \ 4 - a, \ \ \ \# \left\{ i,j,m,n, k, \ell, p, q \right\} \ = \ 8 - b.
\end{align}
Then, we have
\begin{align}
\E \left[ A_{ij} A_{mn} \right] \ &= \ O \left( \frac{d_b}{N} \right)^{2 - \lfloor a/2 \rfloor}, \label{eq:secondmomentestimatesadjacencymatrix} \\
\E \left[ A_{ij} A_{mn} A_{k \ell} A_{pq} \right] \ &= \ O \left( \frac{d_b}{N} \right)^{4 - \lfloor b/2 \rfloor}. \label{eq:fourthmomentestimatesadjacencymatrix}
\end{align}
\end{lemma}
\begin{proof}
We first note $A_{ii} = 0$ for all indices $i$, so we may assume the bounds $a \leq 2, b \leq 4$. Moreover we also note all moments of $A_{ij}$ are equal. Then, the desired estimates \eqref{eq:secondmomentestimatesadjacencymatrix} and \eqref{eq:fourthmomentestimatesadjacencymatrix} follow from Lemma \ref{lemma:pseudoindependenceadjacencymatrix}. We illustrate for the case $a = 1$; suppose, without loss of generality, that $i = m$. Thus, by Lemma \ref{lemma:pseudoindependenceadjacencymatrix}, we have
\begin{align}
\E \left[ A_{ij} A_{in} \right] \ = \ O \left( \frac{d_b}{N} \right) \E A_{ij} \ = \ O \left( \frac{d_b}{N} \right)^2 \ \leq \ O \left( \frac{d_b}{N} \right)^{3/2}.
\end{align}
\end{proof}
We now introduce the following notation for convenience:
\begin{align}
I_1 \ := \ I(ij,mn; A), \ \ \ I_2 \ := \ I(k\ell, pq; A), \\
J_{12} \ := \ J(\{ij, mn\}, \{k\ell, pq\}; A), \ \ \ I_{12} \ := \ I_1 I_2 J_{12}.
\end{align}
%
\begin{remark}
We recall that $I_1 = 0$ and $J_{12} = 0$ with low probability. We now derive the following estimates conditioning on exceptional events with regards to the simple switchings dynamics on $\wt{\Omega}$.
\end{remark}
\begin{lemma} \label{lemma:momentestimateslowprobability}
Let $a,b$ be defined as in Lemma \ref{lemma:momentestimatesadjacencymatrix}. Then, we have
\begin{align}
\E \left[ (A_{ij} A_{mn} + A_{in} A_{mj})(1 - I_1) \right] \ &= \ O \left( \frac{d_b}{N} \right)^{3 - a}, \label{eq:secondmomentestimateslowprobability} \\
\E \left[ (A_{ij} A_{mn} + A_{in} A_{mj})(A_{k \ell} A_{pq} + A_{kq} A_{p \ell}(1 - I_{12}) \right] \ &= \ O \left( \frac{d_b}{N} \right)^{5-b}. \label{eq:fourthmomentestimateslowprobability}
\end{align}
\end{lemma}
\begin{proof}
We first assume $a, b = 0$; for $a, b \neq 0$, the estimates follow directly from Lemma \ref{lemma:momentestimatesadjacencymatrix} and the following inequalities that hold for $a, b \neq 0$:
\begin{align}
3 - a \ \leq \ 2 - \left\lfloor \frac{a}{2} \right\rfloor, \ \ \ 5 - b \ \leq \ 4 - \left\lfloor \frac{b}{2} \right\rfloor.
\end{align}
If $a = 0$, then the event $A_{ij}A_{mn}I_1 = 0$ corresponds to the event that the subgraph restricted to the vertices $(i,j),(m,n)$ is not $1$-regular. In particular, we have the following bound:
\begin{align}
\E \left[A_{ij} A_{mn}(1 - I_1) \right] \ \leq \ \E \left[ A_{ij} A_{mn} \left( A_{in} + A_{mj} \right) \right] \ = \ O \left( \frac{d_b}{N} \right)^3,
\end{align}
where the last bound follows by the assumption $a = 0$ and Lemma \ref{lemma:pseudoindependenceadjacencymatrix}. Because the law of $A_{ij}$ is invariant under relabeling vertices, we deduce \eqref{eq:secondmomentestimateslowprobability}. 

Similarly, if $b = 0$, then the event $A_{ij}A_{mn}A_{k \ell}A_{pq}(1 - I_{12}) = 0$ corresponds to the event that the subgraph restricted to the vertices $(i,j), (m,n), (k, \ell), (p,q)$ is not bipartite, or if it is bipartite, the subgraphs are not $1$-regular. Formally, we have the estimate (conditioning on $b = 0$)
\begin{align}
\E \left[ A_{ij} A_{mn} A_{k \ell} A_{pq} (1 - I_{12}) \right] \ \leq \ \E \left[ A_{ij} A_{mn} A_{k \ell} A_{pq} A_{\Sigma} \right],
\end{align}
where we define
\begin{align}
A_{\Sigma} \ = \ A_{in} + A_{mj} + A_{kq} + A_{p \ell} + A_{i q} + A_{i \ell} + A_{m \ell} + A_{mq} + A_{kj} + A_{kn} + A_{pj} + A_{pn}.
\end{align}
Similarly, we deduce the following bound via Lemma \ref{lemma:pseudoindependenceadjacencymatrix}:
\begin{align}
\E \left[ A_{ij} A_{mn} A_{k \ell} A_{pq} (1 - I_{12}) \right] \ \leq \ O \left( \frac{d_b}{N} \right)^5.
\end{align}
Relabeling vertices and by the assumption $b = 0$, as in the proof of \eqref{eq:secondmomentestimateslowprobability} we deduce \eqref{eq:fourthmomentestimateslowprobability}.
\end{proof}
As a direct consequence of the estimates in Lemma \ref{lemma:momentestimateslowprobability}, we deduce the following averaged estimates.
\begin{lemma}
Suppose $\alpha, \beta$ are defined by the equations
\begin{align}
\# \{i,j\} \ = \ 2 - \alpha, \ \ \ \# \{i,j,k,\ell\} \ = \ 4 - \beta. 
\end{align}
Then, we have the following estimates:
\begin{align}
\frac{1}{N^2} \sum_{(m,n) \in V_b} \ \E \left[ (A_{ij} A_{mn} + A_{in} A_{mj})(1 - I_1) \right] \ &= \ O \left( \frac{d_b}{N} \right)^{3- \alpha}, \label{eq:secondmomentestimateslowprobabilityaverage} \\
\frac{1}{N^4} \sum_{(m,n) \in V_b} \sum_{(p,q) \in V_b} \ \E \left[ (A_{ij} A_{mn} + A_{in} A_{mj})(A_{k \ell} A_{pq} + A_{kq} A_{p \ell} )(1-I_{12})\right] \ &= \ O \left( \frac{d_b}{N} \right)^{5 - \beta}. \label{eq:fourthmomentestimateslowprobabilityaverage}
\end{align}
Moreover, we have
\begin{align}
\frac{1}{N^4} \sum_{(i,j), (m,n) \in V_b} \ \E \left[ (A_{ij} A_{mn} + A_{in} A_{mj})(1 - I_1) \right] \ &= \ O \left( \frac{d_b}{N} \right)^3, \label{eq:secondmomentestimateslowprobabilityaverageaverage} \\
\frac{1}{N^8} \sum_{(i,j), (m,n) \in V_b} \ \sum_{(k, \ell), (p,q) \in V_b} \ \E \left[ (A_{ij} A_{mn} + A_{in} A_{mj})(A_{k \ell} A_{pq} + A_{kq} A_{p \ell} )(1-I_{12})\right] \ &= \ O \left( \frac{d_b}{N} \right)^5. \label{eq:fourthmomentestimateslowprobabilityaverageaverage}. 
\end{align}
\end{lemma}
\begin{proof}
We first note the estimates \eqref{eq:secondmomentestimateslowprobabilityaverageaverage} and \eqref{eq:fourthmomentestimateslowprobabilityaverageaverage} follow from \eqref{eq:secondmomentestimateslowprobabilityaverage} and \eqref{eq:fourthmomentestimateslowprobabilityaverage}, respectively, noting there are $O(N^2)$ pairs $(i,j)$ such that $\alpha = 0$ and $O(N^4)$ pairs of pairs $\{(i,j), (k,\ell)\}$ such that $\beta = 0$.

More generally, for any $\alpha_0 \in [[0,1]]$ (resp. $\beta_0 \in [[0,3]]$), we note there are $O(N^{2 - s})$ (resp. $O(N^{4-s})$) sets $\{i,j\}$ (resp. $\{i,j,k,\ell\}$) such that $\alpha = \alpha_0$ (resp. $\beta = \beta_0$). Thus, by \eqref{eq:secondmomentestimateslowprobability} and \eqref{eq:fourthmomentestimateslowprobability}, we have 
\begin{align}
\frac{1}{N^2} \sum_{(m,n) \in V_b} \ \E \left[ (A_{ij} A_{mn} + A_{in} A_{mj})(1 - I_1) \right] \ &= \ O \left( \frac{d_b}{N} \right)^{3-\alpha} + \sum_{\alpha_0 = 1}^2 \ O(N^{-\alpha_0}) \times O \left( \frac{d_b}{N} \right)^{3-\alpha-\alpha_0} \\
&= \ O \left( \frac{d_b}{N} \right)^{3-\alpha}.
\end{align}
Thus, we derive \eqref{eq:secondmomentestimateslowprobabilityaverage}. Similarly, we may also deduce \eqref{eq:fourthmomentestimateslowprobabilityaverage}.
\end{proof}
We now recall the following seminorm for bounded measurable functions $f$: for $r \geq 1$, define 
\begin{align}
\| f \|_r \ := \ \left( \E |f(A)|^r \right)^{1/r}.
\end{align}
Extending the derivatives, we recall, for $f \in C^k(\mathscr{M}_{\mathbf{e}})$,
\begin{align}
\| \partial^k f \|_r \ := \ \sup_{\theta \in [0,1]^k} \sup_{X \in \mathscr{X}^k} \ \| \partial_{X_1} \ldots \partial_{X_k} f(A + \theta \cdot X) \|_r;
\end{align}
here, we use the notation $X = (X_1, \ldots, X_k) \in \mathscr{X}^k$. 

We now derive the following moment estimates coupled with functions of matrices; the bounds in the following lemma follow from the Holder inequality coupled with the the preceding moment estimates. These bounds will be important in bounding the error terms $R_1$ and $R_2$. 
\begin{lemma} \label{lemma:holderbounds}
Fix $\e > 0$ and suppose $r = r(\e) \gg 1$ is sufficiently large depending on the the parameter $\e$. Define the parameters $\alpha, \beta$ by the equations
\begin{align}
\# \{ i,j \} \ = \ 2 - \alpha, \ \# \{i,j,k,\ell\} \ = \ 4 - \beta.
\end{align}
For any $f \in C^0(\mathscr{M}_{\mathbf{e},0})$ with $\| f \|_r \leq 1$, we have the following estimates:
\begin{align}
\frac{1}{N^4} \sum_{(i,j),(m,n) \in V_b} \ \E \left[ f(A) A_{ij} A_{mn} \overline{I_1} \right] \ &= \  O \left( \frac{d_b}{N} \right)^{3 - \e}, \\
\frac{1}{N^4} \sum_{(i,j),(m,n) \in V_b} \ \E \left[ f(A) A_{ij} A_{mn} \right] \ &= \ O \left( \frac{d_b}{N} \right)^{2 - \lfloor \alpha/2 \rfloor - \e}, \label{eq:examplecalculation} \\
\frac{1}{Nd_w} \sum_{(m,n) \in V_b} \ \E \left[ f(A) \left( A_{ij} A_{mn} - A_{in} A_{mj} \right) \overline{I_1} \right] \ &= \ O \left( \frac{d_b}{N} \right)^{2 - \alpha - \varepsilon}, \\
\frac{1}{N d_w} \sum_{(m,n) \in V_b} \ \E \left[ f(A) A_{ij} A_{mn} \right] \ &= \ O \left( \frac{d_b}{N} \right)^{1 - \varepsilon}, \\
\frac{1}{(Nd_w)^2} \sum_{(m,n),(p,q) \in V_b} \ \E \left[ f(A) \left( A_{ij} A_{mn} - A_{in} A_{mj} \right) \left( A_{k\ell} A_{pq} - A_{kq} A_{p \ell} \right) \overline{I_{12}} \right] \ &= \ O \left( \frac{d_b}{N} \right)^{3 - \beta - \varepsilon}, \\
\frac{1}{(Nd_w)^2} \sum_{(m,n),(p,q) \in V_b} \ \E \left[ f(A) \left( A_{ij} A_{mn} - A_{in} A_{mj} \right) \left( A_{k\ell} A_{pq} - A_{kq} A_{p \ell} \right) \right] \ &= \ O \left( \frac{d_b}{N} \right)^{4 - \lfloor \beta/2 \rfloor - \e}.
\end{align}
Here, for an indicator random variable $\chi$ corresponding to an event $\Xi$, we denote by $\overline{\chi} = 1 - \chi$ the indicator random variable of the complement of $\Xi$.
\end{lemma}
\begin{proof}
We prove \eqref{eq:examplecalculation}; the other bounds follow analogously. By the Holder inequality with respect to the expectation $\E(\cdot)$, we see
\begin{align}
\E \left[ f(A) A_{ij} A_{mn} \right] \ \leq \ \|f\|_r \left( \E \left[ A_{ij} A_{mn} \right] \right)^{1/q} \ \leq \ \left( \E A_{ij} A_{mn} \right)^{1/q},
\end{align}
where $q^{-1} = 1 - r^{-1}$. Again by the Holder inequality with respect to the summation over indices $(i,j), (m,n) \in V_b$, we see
\begin{align}
\frac{1}{N^4} \sum_{(i,j),(m,n)} \ \left( \E A_{ij} A_{mn} \right)^{1/q} \ &\leq \ \frac{1}{N^4} (MN)^{2r} \left( \sum_{(i,j),(m,n) \in V_b} \ \E A_{ij} A_{mn} \right)^{1/q} \\
&= \ \left( \sum_{(i,j),(m,n) \in V_b} \ \frac{1}{N^4} \E A_{ij} A_{mn} \right)^{1 - r^{-1}}.
\end{align}
By Lemma \ref{lemma:momentestimatesadjacencymatrix}, this upper bound is also bounded by the following:
\begin{align}
O \left( \frac{d_b}{N} \right)^{2 - \lfloor \alpha/2 \rfloor - C(\alpha) r^{-1}} \ = \ O \left( \frac{d_b}{N} \right)^{2 - \lfloor \alpha/2 \rfloor - \e},
\end{align}
where the equality holds by choosing $r(\e) \gg 1$ and the trivial bound $\alpha = O(1)$. This proves \eqref{eq:examplecalculation}. 
\end{proof}
We now prove the last estimate bounding expectations of function $f: \Omega \to \R$ with coefficients given by the fluctuation terms $A_{ij} - d_b/N$ in the expansions \eqref{eq:efexpansionij} and \eqref{eq:efexpansionmn}. These, along with Lemma \ref{lemma:holderbounds}, will be important in bounding the error terms $R_3, \ldots, R_5$. 
\begin{prop} \label{prop:efestimates}
Fix any $\e > 0$, and suppose $r = r(\e) \gg_\e 1$ is sufficiently large depending on $\e$. Moreover, define $\alpha, \beta$ by the equations
\begin{align}
\# \left\{ i, j \right\} \ = \ 2 - \alpha , \ \# \{ i, j, m, n \} \ = \ 4 - \beta.
\end{align}
For any $f \in C^0(\mathscr{M}_{\mathbf{e},0})$, we have
\begin{align}
\E \left[ f(A) \left( A_{ij} - \frac{d_b}{N} \right) \right] \ &= \ O \left( \frac{d_b}{N} \right)^{1 - \varepsilon} \| \partial f \|_r \ + \ O \left( \frac{d_b}{N} \right)^{2 - \alpha - \varepsilon} \| f \|_r, \label{eq:secondorderef} \\
\E \left[ f(A) \left( A_{ij} - \frac{d_b}{N} \right) \left( A_{mn} - \frac{d_b}{N} \right) \right] \ &= \ O \left( \frac{d_b}{N} \right)^{2 - \e} \| \partial^2 f \|_r \ + \ O \left( \frac{d_b}{N} \right)^{3 - \beta - \e} \| f \|_r. \label{eq:lowestorderef}
\end{align}
\end{prop}
\begin{proof}
We employ the following identity, which follows from the averaging $\sum_{m,n} A_{mn} = 2Nd_w$, $\sum_n A_{in} = d_b$ and $\sum_m A_{mj} = d_w$. 
\begin{align}
f(A) \left( A_{ij} - \frac{d_b}{N} \right) \ &= \ \frac{1}{2Nd_w} f(A) \sum_{(m,n) \in V_b} \ \left( A_{ij} A_{mn} - A_{in} A_{mj} \right).
\end{align}
In particular, writing $1 = I_1 + \overline{I_1}$, we have, by Lemma \ref{lemma:holderbounds},
\begin{align}
\E \left[ f(A) \left( A_{ij} - \frac{d_b}{N} \right) \right] \ &= \ \frac{1}{2Nd_w} \E \left[ f(A) \sum_{(m,n) \in V_b} \left( A_{ij} A_{mn} - A_{in} A_{mj} \right) I_1 \right] \ + \ O \left( \frac{d_b}{N} \right)^{2 - \alpha - \e} \|f \|_r.
\end{align}
Thus, it suffices to bound the first average containing the factor $I_1$. We first note we may assume $\alpha = 0$ else the first average vanishes. Thus, we see $I_1(A) = I_1(T_S(A))$ with $S$ the subgraph whose edges are given by $\{ ij \}$ and $\{ mn \}$. With this and the invariance of the uniform measure under $T_S(A)$, we see
\begin{align}
\E \left[ f(A) \left( A_{ij} - \frac{d_b}{N} \right) \right] \ = \ \frac{1}{2N d_w} \sum_{(m,n) \in V_b} \ \E \left[ \left( f(A) - f(A - \xi_{ij}^{mn}) \right) A_{ij}A_{mn} I_1 \right].
\end{align}
With the following Taylor estimate:
\begin{align}
\left| f(A) - f(A - \xi_{ij}^{mn}) \right| \ \leq \ \sup_{\theta \in [0,1]} \sup_{X \in \mathscr{X}} \ \left| \partial_X f(A + \theta X) \right|,
\end{align}
as well as Lemma \ref{lemma:holderbounds}, we deduce \eqref{eq:secondorderef}. To prove \eqref{eq:lowestorderef}, we appeal to the identity
\begin{align}
f(A) \left( A_{ij} - \frac{d_b}{N} \right) \left( A_{k \ell} - \frac{d_b}{N} \right) \ = \ \frac{1}{(2Nd_w)^2} f(A) \sum_{(m,n), (p,q) \in V_b} \ \left( A_{ij} A_{mn} - A_{in} A_{mj} \right) \left( A_{k \ell} A_{pq} - A_{kq} A_{p \ell} \right).
\end{align}
Writing $1 = I_{12} + \overline{I}_{12}$, we have, by Lemma \ref{lemma:holderbounds},
\begin{align}
\E \left[ f(A) \left( A_{ij} - \frac{d_b}{N} \right) \left( A_{k \ell} - \frac{d_b}{N} \right) \right] \ = \ &\frac{1}{(2Nd_w)^2} \sum_{(m,n), (p,q) \in V_b}  \E \left[ \wt{G}_{ij,mn,k\ell,pq} I_{12} \right] \ + \ O \left( \frac{d_b}{N} \right)^{3 - \beta - \varepsilon} \| f \|_r,
\end{align}
where we define
\begin{align}
\wt{G}_{ij,mn,k\ell,pq} \ = \ f(A) \left( A_{ij} A_{mn} - A_{in} A_{mj} \right) \left( A_{k \ell} A_{pq} - A_{kq} A_{p \ell} \right).
\end{align}
Similarly, we may assume $\beta = 0$, else the first averaging term containing the factor $I_{12}$ vanishes. Again, appealing to the invariance of the uniform measure under $T_S(A)$ and the identities $I_{12}(A) = I_{12}(T_{S_1, S_2}(A))$ with $S_1, S_2$ the subgraphs containing the vertices $\{ ij, mn \}$ and $\{k \ell, pq\}$ respectively, we see
\begin{align}
\E \left[ f(A) \left( A_{ij} - \frac{d_b}{N} \right) \left( A_{k \ell} - \frac{d_b}{N} \right) \right] \ = \ \frac{1}{(2 N d_w)^2} \sum_{(m,n), (p,q) \in V_b} \ \E \left[ G_{ij,mn,k\ell,pq}(A) A_{ij}A_{mn} I_{12} \right],
\end{align}
where we define
\begin{align}
G_{ij,mn,k\ell,pq}(A) \ := \ f(A) - f(A - \xi_{ij}^{mn}) - f(A - \xi_{k \ell}^{pq}) + f(A - \xi_{k \ell}^{pq} - \xi_{ij}^{mn}).
\end{align}
Appealing to the Taylor estimate
\begin{align}
\left| G_{ij,mn,k\ell,pq}(A) \right| \ \leq \ \sup_{\theta_1, \theta_2 \in [0,1]} \sup_{X_1, X_2 \in \mathscr{X}} \ \left| \partial_{X_1} \partial_{X_2} f(A + \theta X_1 + \theta X_2) \right|,
\end{align}
as well as Lemma \ref{lemma:holderbounds}, we also deduce \eqref{eq:lowestorderef}.
\end{proof}
We now prove Proposition \ref{prop:errorestimatessobolev}. By Lemma \ref{lemma:holderbounds} and Proposition \ref{prop:efestimates}, we have
\begin{align}
\E R_1 \ &= \ O \left( \frac{d_b}{N} \right)^{2 - \e} \| \partial f \|_r, \\
\E R_2 \ &= \ O \left( \frac{d_b}{N} \right)^{1 - \e} \| \partial^3 f \|_r, \\
\E R_3 \ &= \ O \left( \frac{d_b}{N} \right)^{1 - \e} \| \partial^3 f \|_r, \\
\E R_4 \ &= \ O \left( \frac{d_b}{N} \right)^{2 - \e} \| \partial^2 f \|_r \ + \ O \left( \frac{d_b}{N} \right)^{1 - \e} \| \partial^3 f \|_r, \\
\E R_5 \ &= \ O \left( \frac{d_b}{N} \right)^{2 - \e} \| \partial^2 f \|_r \ + \ O \left( \frac{d_b}{N} \right)^{1 - \e} \| \partial^4 f \|_r.
\end{align}
Thus, by definition of the parameter $D$ and the change of variables $f(A) = F(H)$ with 
\begin{align}
\partial^k f \ = \ d_b^{-k/2} \partial^k F,
\end{align}
we have
\begin{align}
\sum_{i = 1}^5 \ \E R_i \ = \ O \left( D^{-1/2} N^{-1 + \e} \right) \sum_{i = 1}^4 \ \| \partial^i F \|_{r(\e),0}.
\end{align}
This concludes the proof of Proposition \ref{prop:errorestimatessobolev}. \QED
%
%
%


\end{document}